\documentclass[reqno]{amsart}
\usepackage{amssymb,harpoon}

\usepackage[colorlinks=true]{hyperref}
\hypersetup{urlcolor=blue, citecolor=red}

\usepackage[colorinlistoftodos]{todonotes}

\usepackage{enumitem}
\usepackage{calligra}
\usepackage{cancel}
\setitemize{wide}

\newtheorem{Lemma}{Lemma}[section]
\newtheorem{Theorem}[Lemma]{Theorem}
\newtheorem{Proposition}[Lemma]{Proposition}
\newtheorem{Corollary}[Lemma]{Corollary}
\newtheorem{Remark}[Lemma]{Remark}

\def\C {\mathbb C}

\def\R {\mathbb R}
\def\Z {\mathbb Z}

\newcommand{\cC}{{\mathcal C}}
\newcommand{\cD}{{\mathcal D}}
\newcommand{\cF}{{\mathcal F}}
\newcommand{\cH}{{\mathcal H}}
\newcommand{\cI}{{\mathcal I}}
\newcommand{\cL}{{\mathcal L}}
\newcommand{\cS}{{\mathcal S}}
\newcommand{\V}{{\mathcal V}}
\newcommand{\cO}{{\mathcal O}}

\newcommand{\dprod}[2]{\langle{#1},{#2}\rangle}

\def\uh{ {\overline u} }
\def\ua{ {\underline u} }
\def\vh{ {\overline v} }

\def\wh{ {\overline w} }
\def\wa{ {\underline w} }

\def\abar {{\overline{a}}}
\def\zbar {{\overline{z}}}

\def\Bh{\overset{\rightharpoonup}{\vphantom{a}\smash{\mathcal{B}}}}
\def\Ba{\overset{\leftharpoonup}{\vphantom{a}\smash{\mathcal{B}}}}
\def\Dh{\overset{\rightharpoonup}{\vphantom{a}\smash{\mathcal{D}}}}
\def\Da{\overset{\leftharpoonup}{\vphantom{a}\smash{\mathcal{D}}}}
\def\uh{\overset{\rightharpoonup}{\vphantom{a}\smash{u}}}
\def\ua{\overset{\leftharpoonup}{\vphantom{a}\smash{u}}}
\def\vh{\overset{\rightharpoonup}{\vphantom{a}\smash{v}}}

\def\wh{{\overset{\rightharpoonup}{\vphantom{a}\smash{w}}}}
\def\wa{{\overset{\leftharpoonup}{\vphantom{a}\smash{w}}}}

\newcommand{\diag}{\operatorname{diag}}
\newcommand{\<}{\langle}
\renewcommand{\>}{\rangle}
\renewcommand{\(}{\left(}
\renewcommand{\)}{\right)}

\renewcommand{\Im}{\operatorname{Im}}
\renewcommand{\Re}{\operatorname{Re}}

\newcommand{\p}{\partial}
\newcommand{\Vol}{\operatorname{Vol}}
\newcommand{\sgn}{\operatorname{sgn}}
\newcommand{\id}{\operatorname{Id}}
\newcommand{\supp}{\operatorname{supp}}
\newcommand{\const}{\operatorname{const}}
\newcommand{\sing}{\operatorname{sing}}

\begin{document}
\title[The attenuated ray transform on pairs]{Inversion formulas and range characterizations for the attenuated geodesic ray transform}

\author{Yernat M. Assylbekov}
\address{Department of Mathematics, Northeastern University, Boston, MA 02115, USA}
\email{y\_assylbekov@yahoo.com}

\author{Fran\c{c}ois Monard}
\address{Department of Mathematics, University of California, Santa Cruz, CA 95064, USA}
\email{fmonard@ucsc.edu}

\author{Gunther Uhlmann}
\address{Department of Mathematics, University of Washington, Seattle, WA 98195-4350, USA. 
Department of Mathematics and statistics University of Helsinki, Box 68, Helsinki, 00014, Finland. 
Institute for Advanced Study, Hong Kong University of Science and Technology, Hong Kong SAR}
\email{gunther@math.washington.edu}
\maketitle

\begin{abstract}
    We present two range characterizations for the attenuated geodesic X-ray transform defined on pairs of functions and one-forms on simple surfaces. Such characterizations are based on first isolating the range over sums of functions and one-forms, then separating each sub-range in two ways, first by implicit conditions, second by deriving new inversion formulas for sums of functions and one-forms.  
\end{abstract}


\section{Introduction}

Let $(M,g)$ be a smooth compact oriented Riemannian surface with boundary $\p M$, with unit tangent bundle $SM:=\{(x,v)\in TM: |v|_{g(x)}=1\}$ and inward/outward boundaries 
\begin{align*}
\p_\pm SM = \{(x,v)\in SM:x\in\p M,\  \pm \<v,\nu_x\>_{g(x)} \ge 0\},
\end{align*}
where $\nu_x$ is the unit inward normal at $x\in \partial M$. Denote $\varphi_t:SM\to SM$ the geodesic flow, written as $\varphi_t(x,v) = (\gamma_{x,v}(t), \dot \gamma_{x,v}(t))$ and defined for $-\tau(x,-v)\le t\le \tau(x,v)$, where $\tau(x,v)$ is the first exit time of the geodesic starting at $(x,v)$. Throughout the paper, we assume that $(M,g)$ is {\em simple}, meaning that the boundary is strictly convex and that any two points on the boundary are joined by a unique minimizing geodesic. In particular, this implies that $(M,g)$ is simply connected and that $\tau(x,v)$ is bounded on $SM$ (i.e., $(M,g)$ is non-trapping). For $a\in C^\infty(M,\C)$, the object of study is the {\em attenuated geodesic ray transform} $I_a:C^\infty(SM) \to C^\infty(\partial_+ SM)$ defined for $f\in C^\infty(SM)$ as 
\begin{align}
    I_a f (x,v) = \int_0^{\tau(x,v)} f(\varphi_t(x,v))\exp \left(\int_0^t a(\gamma_{x,v}(s)\ ds\right)\ dt, \qquad (x,v) \in \partial_+ SM.
    \label{eq:attenuated}
\end{align}
The present article aims at providing range characterizations for this transform over pairs of functions and one-forms, or equivalently, when the integrand $f$ above takes the form $f(x,v) = f_0 (x) + \alpha_x(v)$ for $[f_0,\alpha]$ a pair of a function and a one-form. 
As the transform above models some medical imaging modalities such as Computerized Tomography and Ultrasound Doppler Tomography in media with variable refractive index, range characterizations are useful to project noisy data onto the range of a given measurement operator before inverting for the unknown ($f_0$ or $\alpha$ here). In media with constant refractive index, modelled by the Euclidean metric in the parallel geometry, the problem was extensively studied \cite{Novikov2002,Natterer2001a,BomanStroemberg2004,Finch2004,StroembergAndersson2012}, and the range characterization was already a challenging issue yet to be solved \cite{Novikov2002a}. Recently, range characterizations for the attenuated transform on convex Euclidean domains were provided in terms of Hilbert transforms with respect to A-analytic function theory {\it \`a la} Bukhgeim, treating the case of functions \cite{Sadiq2013}, vector fields \cite{Sadiq2014} and two-tensors \cite{Sadiq2015}, though such results are limited to Euclidean settings as A-analytic function theory has not yet been developed on general surfaces. 

In the case of manifolds with no symmetries, parallel geometry does not exist and one must work with fan-beam coordinates. The scalar case has been studied in \cite{SaU,Monard2015} in the geodesic case, and in \cite{Kazantsev2007} in the Euclidean, fan-beam case, mainly focused on injectivity, stability and inversion procedures.

On to range characterizations, the first one in terms of boundary operators was provided by Pestov and Uhlmann in \cite{PeU2}, later generalized to the case of transport with unitary connection, with further applications to the range characterization of the unattenuated transform over higher-order tensors \cite{PSU3}. Recently in \cite{Monard2016}, the range characterization in \cite{PeU2} was proved by the second author to be a generalization of the classical moment conditions in the Euclidean setting. 

In the approach coming from \cite{PeU2}, there is a boundary operator $P$ which only depends on the scattering relation and the fiberwise Hilbert transform, and which characterizes the unattenuated transform over functions and one-forms. Further splitting of $P$ into the sum $P_+ + P_-$ allows to separate ranges over functions and one-forms. A major challenge in the attenuated case is that, despite the fact that a similar operator exists for the ray transform over pairs (a fact which is one of the first features of this article), the splitting mentioned above is no longer straightforward. We then propose two approaches to separate the sub-ranges within the range over pairs. 
 
The first approach is an implicit description given by adding constraints on the preimage by the $P$ operator above, while the second one relies on inversion formulas for each term of the pair, from the data of both. 

In \cite{Monard2015}, the second author provides inversion formulas for the attenuated ray transforms for functions and vector fields, including one which takes the form of a Fredholm equation, in which the operator may depend on the attenuation coefficient. The formulas presented here allows inversions for pairs (function + one-form) modulo natural obstructions. Moreover, the integrands can be supported up to the boundary. The formulas are exact provided that one can invert the unattenuated transform over functions and solenoidal vector fields. In that regard, the approach does not suffer from whether the attenuation is too low or too high as in \cite{Monard2015}. Additionally, it is generalized to complex-valued attenuations, which requires using both holomorphic and antiholomorphic integrating factors, as in the first inversion procedure presented in \cite{SaU}. An additional tool which is introduced and allows to extract information in a systematic fashion, is a way to turn transport solutions with holomorphic right-hand sides into holomorphic solutions themselves, by manipulating their boundary values. In some sense, this operation is to be understood as a change in the qualitative features of the solutions by ``data'' processing. 

We now state the main results and give an outline of the remainder of the article in the next section.

\section{Statements of main results} \label{sec:main}

In what follows, for $F$ some function space ($C^k$, $L^p$, $H^k$, etc.), we denote by $\cF(M,\C)$ the corresponding space of pairs $[\alpha,f]$ with $\alpha$ a $1$-form and $f$ a function on $M$. In particular, $\cC^\infty(M,\C)$ is the space of pairs $[\alpha,f]$, with $\alpha\in C^\infty(\Lambda^1(M),\C)$ and $f\in C^\infty(M,\C)$. Then $\cI_a$ denotes the restriction of $I_a$ to $\cC^\infty(M,\C)$:
$$
\cI_{a}[\alpha,f](x,v):=I_{a}^1\alpha(x,v)+I_{a}^0f(x,v),\quad (x,v)\in\p_+SM,
$$
where $I_a^1$ and $I_a^0$ are the restrictions of $I_a$ to $1$-forms and functions on $M$, respectively.

Let $X(x,v) = \frac{d}{dt}|_{t=0} \varphi_t(x,v)$ denote the generator of the geodesic flow of $g$, a global section of $T(SM)$. Here and below, for a given $w\in C^\infty(\p_+ SM,\C)$ we denote by $w^\sharp:SM\to \C$ the unique solution to the transport equation
\begin{align*}
    Xw^\sharp+aw^\sharp=0 \quad (SM),\qquad w^\sharp\big|_{\p_+ SM}=w.    
\end{align*}
We then define $Q_a:C(\partial_+ SM, \C) \to C(\partial SM, \C)$, by $Q_a w := w^\sharp|_{\partial SM}$. $Q_a$ takes the expression
\begin{align}
    Q_a w (x,v) = \begin{cases} 
	w(x,v)\quad & (x,v)\in\p_+SM,\\ 
	\exp \left( - \int_{-\tau(x,-v)}^0 a(\gamma_{x,v}(t)) \right) w(\alpha(x,v)) \quad&(x,v)\in\p_-SM,
    \end{cases}
    \label{eq:Qa}
\end{align}
where $\alpha$ denotes the {\em scattering relation}\footnote{Throughout the paper, $\alpha$ may denote either the scattering relation, or a general one-form, though which occurence it is should appear clear from the context.} defined in Section \ref{sec:prelim}. As $Q_a w$ may only be continuous even when $w\in C^\infty(\partial_+ SM)$, we define 
\[ \cS_{a}^\infty(\p_+ SM,\C) := \{ w\in C^\infty(\partial_+ SM, \C),\ Q_a w\in C^\infty(\partial SM)  \}.  \]

We also introduce the operator $B_a:C(\p SM,\C)\to C(\p_+SM,\C)$ by
\begin{align}
    B_{a}u(x,v):=\exp\({\int_0^{\tau(x,v)}a(\gamma_{x,v}(t))\,dt}\)u\circ\alpha(x,v)-u(x,v),\quad (x,v)\in \p_+ SM.  
    \label{eq:Ba}
\end{align}

Next, we introduce the operator $P_{a}:\cS_{a}^\infty(\p_+ SM,\C)\to C^\infty(\p_+ SM,\C)$ defined by $P_{a}:=B_{a} H Q_{a}$, where $H$ is the fiberwise Hilbert transform, defined in Section \ref{sec:prelim}. Clearly the operator $P_{a}$ is completely determined by the scattering relation $\alpha$ and the unattenuated ray transform of $a$. The first main result of the paper is that the operator $P_a$ characterizes the ray transform $\cI_a$ over pairs. 

\begin{Theorem}\label{main} Let $(M,g)$ be a simple surface and let $a\in C^\infty(M,\R)$. Then a function $u\in C^\infty(\p_+ SM,\C)$ belongs to the range of $\cI_{a}$ if and only if $u=P_{a}w$ for some $w\in\cS_{a}^\infty(\p_+ SM,\C)$.
\end{Theorem}

As mentioned in the introduction, the corresponding operator $P:= P_0$, first introduced in \cite{PeU2}, splits into two operators $P_+$ and $P_-$ which characterize the ranges of $I^0$ and $I^1$ separately. In the attenuated case, such a splitting is no longer obvious. Sitting within the range of $\cI_a$, a first range characterization for $I_a^0$ and $I_a^1$ can be obtained by adding conditions on the preimage by $P_a$. Before stating the result, we introduce some notations.

Since $M$ is oriented there is a circle action on the fibres of $SM$ with infinitesimal generator $V$ called the vertical vector field. For any two functions $u,v:SM\to\C$ define an inner product:
$$
\langle u,v\rangle_{L^2(SM)}=\int_{SM} u\overline v\,d\Sigma^3,
$$
where $d\Sigma^3$ is the Liouville measure of $g$ on $SM$. The space $L^2(SM,\C)$ decomposes orthogonally as a direct sum
$$
L^2(SM,\C)=\bigoplus_{k\in\mathbb Z}H_k
$$
where $H_k$ is the eigenspace of $-iV$ corresponding to the eigenvalue $k$. Any function $u\in C^{\infty}(SM,\C)$ has a Fourier series expansion
$$
u=\sum_{k=-\infty}^{\infty}u_k, \qquad u_k\in \Omega_k:=C^{\infty}(SM,\C)\cap H_k.
$$
In particular, $u\mapsto u_0$ and $u\mapsto u_{-1}+u_1$ are the projections of functions on $SM$ onto functions and $1$-forms on $M$, respectively; see Section~\ref{3.2}.

\begin{Theorem}\label{thm:charac1}
Let $(M,g)$ be a simple surface and let $a\in C^\infty(M,\C)$. The following range characterizations hold:
\begin{enumerate}
    \item A function $u\in C^\infty(\p_+ SM,\C)$ belongs to the range of $I^0_{a}$ if and only if $u=P_{a}w$ for some $w\in\mathcal S_{a}^\infty(\p_+ SM,\C)$ such that $w^\sharp_0=0$.
    \item A function $u\in C^\infty(\p_+ SM,\C)$ belongs to the range of $I^1_{a}$ acting on solenoidal one-forms if and only if $u=P_{a}w$ for some $w\in\mathcal S_{a}^\infty(\p_+ SM,\C)$ such that $w^\sharp_{-1}+w^\sharp_{1}=dp$ for some $p\in C^\infty(M,\C)$.
\end{enumerate}
\end{Theorem}

We now derive reconstruction formulas for pairs, which in turn yield a second range characterization. Since the transform over pairs $\cI_a$ has a kernel (namely, the ``$a$-potential'' pairs), it is first useful to change its domain in such a way which makes it injective without altering its range. To this end, we make our way in Theorem \ref{thm:secondrep} and Lemma \ref{lem:decomp} and \ref{lem:4decomp}, into proving that any element $\cD \in \text{Range }\cI_a$ decomposes uniquely as 
\begin{align*}
    \cD = \cI_a [\star d h_0 + \omega_1 + \omega_{-1}, f] = I_a^0 f + I_a^\perp h_0 + I_a^{+1} \omega_1 + I_a^{-1} \omega_{-1},
\end{align*}
with $f\in C^\infty(M)$, $h_0 \in C^\infty_0 (M)$ and $\omega_{\pm 1} \in \ker^{\pm 1} \eta_\mp$ are holomorphic and anti-holomorphic one-forms. Moreover, $\cD = 0$ if and only if $f, h_0, \omega_1$ and $\omega_{-1}$ vanish identically. This suggests that the quadruple $(f,h_0,\omega_1,\omega_{-1})$ can be reconstructed from $\cD$, and we proceed to provide reconstruction formulas for each term in Section \ref{sec:inversion}. We first reconstruct $\omega_1$ and $\omega_{-1}$ from $\cD$ in Theorem \ref{thm:holoterms}, and in turn, explain how to remove $\cI_a[\omega_1+\omega_{-1}, 0]$ from $\cD$. This is done using Hilbert bases of square integrable harmonic one-forms, combined with integration by parts on $SM$ using appropriate adjoint transport solutions with one-sided fiber-harmonic content. In fact, the reconstruction of the terms $\omega_{\pm 1}$ is new even in the unattenuated case, for which the first inversion formulas for one-forms appearing in \cite{PeU2} only treated one-forms $\omega=\star d h$ with $h|_{\partial M} = 0$.

After reconstructing $\omega_1$ and $\omega_{-1}$, it remains to reconstruct $(f,h_0)$ from $\cI_a[\star dh_0,f]$. As a means to obtain exact reconstruction formulas (i.e., not up to Fredholm errors), we first construct in section \ref{sec:holomorphization} a ``holomorphization operator'' $\Bh:C^\infty(\partial SM) \to C^\infty(\partial_+ SM)$ (see Theorem \ref{thm:holomorphization} for details) such that if the equation $Xu = -f$ holds with $f$ holomorphic, then the function $\uh = u - (\Bh (u|_{\partial SM})_\psi)$ is a holomorphic solution of $X\uh = -f$ with $\uh_0$ constant. An antiholomorphization counterpart $\Ba$ is also defined there. Such operators, which allow to extract holomorphic and antiholomorphic contents at will, together with the use of so-called holomorphic and anti-holomorphic integrating factors first defined in \cite{SaU}, are key to deriving the following reconstruction formulas, which we prove in Section \ref{sec:recons2}. See Section \ref{3.2} for a definition of the Guillemin-Kazhdan operators $\eta_\pm$ appearing below. 

\begin{Theorem}\label{thm:intro} Let $(M,g)$ a simple surface and $a\in C^\infty(M,\C)$. Define $\wh$ and $\wa$ smooth holomorphic and antiholomorphic, odd, solutions of $X\wh = X\wa = -a$, and let $\Bh$ and $\Ba$ as in Theorem \ref{thm:holomorphization} and Corollary \ref{cor:antiholomorphization}. Then the functions $(h_0,f)\in C^\infty_0(M)\times C^\infty(M)$ can be reconstructed from data $\cI := \cI_{a} [\star dh_0,f]$ (extended by zero on $\partial_- SM$) via the following formulas:
    \begin{align*}
	f &= -\eta_+ \Dh_{-1} - \eta_- \Da_1 - \frac{a}{2} \left( \Dh_0 + \Da_0 + i(g_+ - g_-) \right), \\
	h_0 &= \frac{1}{2} (g_+ + g_-) - \frac{i}{2} (\Dh_0 - \Da_0),
    \end{align*}
    where we have defined $\Dh := e^\wh (\Bh(\cI e^{-\wh}|_{\partial SM}))_\psi$, $\Da := e^\wa (\Ba(\cI e^{-\wa}|_{\partial SM}))_\psi$, and where $g_\pm \in \ker^0 \eta_\pm$, uniquely characterized by their boundary conditions
    \begin{align*}
	g_+|_{\partial M} = -i (\cI - \Dh|_{\partial SM})_0, \qquad g_-|_{\partial M} = i (\cI - \Da|_{\partial SM})_0. 
    \end{align*}
\end{Theorem}

The reconstruction formulas above then allow to construct in \eqref{eq:Ppm1} and \eqref{eq:P0perp} explicit linear, idempotent operators $P_{a,0}, P_{a,\perp}, P_{a,\pm 1}:\text{Range }\cI_a\to \text{Range }\cI_a$, such that 
\begin{align*}
  P_{a,\pm 1} \cD = I_a^{\pm 1} \omega_{\pm 1}, \qquad P_{a,0} \cD = I_a^0 f, \qquad P_{a,\perp} \cD = I_a^\perp h_0.
\end{align*}

Such operators allow to establish the following range characterization:

\begin{Theorem}\label{thm:charac2} Let $(M,g)$ a simple surface and let $a\in C^\infty(M,\C)$. Then the following hold: 
    \begin{enumerate}
	\item[$(i)$] A function $u\in C^\infty(\partial_+ SM, \C)$ belongs to the range of $I_a^0$ if and only if $u = P_a w$ for some $w\in \cS_a^\infty (\partial_+ SM,\C)$ and $P_{a, 1} u = P_{a,-1} u = P_{a,\perp} u = 0$. 
	\item[$(ii)$] A function $u\in C^\infty(\partial_+ SM, \C)$ belongs to the range of $I_a^1$ acting on solenoidal one-forms if and only if $u = P_a w$ for some $w\in \cS_a^\infty (\partial_+ SM,\C)$ and $P_{a,0} u = 0$. 
    \end{enumerate}    
\end{Theorem}

This characterization is of practical relevance as it allows to project noisy data onto the range of $I_a^0$ or $I_a^1$ acting on solenoidal one-forms using explicit operators, before inversion.\medskip

\noindent{\bf Outline and roadmap of proofs.} We first study the space of pairs [one-form, function] in Section \ref{sec:pairs}, on which the operator $\cI_a$ is defined. Proving Theorem \ref{main} is based on the factorization $-2\pi P_a = \cI_a \left[\begin{smallmatrix} 0 & \star d \\ \star d & 0 \end{smallmatrix}\right] \cI_{-\overline{a}}^*$, which completes the proof once the surjectivity of $\cI_{-\overline{a}}^*$ is proved in appropriate functional settings. Such a surjectivity mainly relies on the injectivity of $\cI_{-\overline{a}}$ (\cite[Theorem 1.2]{SaU}), and is based on pseudodifferential arguments on a slightly extended surface. Theorem \ref{thm:charac1} then follows by finding the appropriate additional conditions which characterize each sub-range. 

On to the proof of Theorem \ref{thm:charac2}, we first explain in Section \ref{sec:viewpoints} how to change the domain $\cI_a$ in a way which makes it injective, in particular via the mapping $(f,h_0,\omega_1,\omega_{-1})\mapsto \cI_a[\star d h_0 + \omega_1 + \omega_{-1},f]$. Section \ref{sec:inversion} then explains how to reconstruct each term: we first reconstruct $\omega_1$ and $\omega_{-1}$ in Section \ref{sec:recons1}; then introduce holomorphization operators in Section \ref{sec:holomorphization}; finally, we provide reconstruction formulas for $(f,h_0)$ in Section \ref{sec:recons2}. In both sections \ref{sec:recons1} and \ref{sec:recons2}, we explain the implications of such inversions on the ability to construct projection operators for Theorem \ref{thm:charac2}. 

\section{Preliminaries}\label{sec:prelim}

\subsection{Scattering relation and transport equations}\label{2.1}
Recall that for $(x,v)\in SM$, $\tau(x,v)$ denotes the first non-negative exit time $\tau(x,v)$ of the geodesic $\gamma_{x,v}$, with $x=\gamma_{x,v}(0)$, $v=\dot\gamma_{x,v}(0)$. The {\it scattering relation} is the map $\alpha:\p SM \to\p SM $ defined as
$$
\alpha(x,v) = \varphi_{\pm \tau(x,\pm v)}(x,v), \qquad (x,v)\in \partial_{\pm} SM,
$$
Since $(M,g)$ is assumed to be simple, by \cite[Lemma~4.1.1]{Sh} we conclude that the scattering relation $\alpha$ is diffeomorphism and $\alpha^2=\id$.

The attenuated ray transform \eqref{eq:attenuated} can be realized as the trace on $\partial_+ SM$ of the solution $u:SM\to \C$ to the following transport equation on $SM$,
$$
Xu+au=-f\quad (SM),\qquad u|_{\p_-SM}=0,
$$
where $f\in C^\infty(SM)$ represents the ``source term''. This equation has a unique solution $u^f$, since on any fixed geodesic the transport equation is an ODE with zero initial condition and an integral expression gives us that $u|_{\partial_+ SM}$ matches \eqref{eq:attenuated}. For $w\in C^\infty(\p_+SM,\C^n)$ given, let us denote $w_\psi (x,v) := w(\varphi_{-\tau(x,-v)}(x,v))$ the unique solution $u$ to the transport problem
$$
Xu=0 \quad (SM),\qquad u\big|_{\p_+SM}=w.
$$
For $a\in C^\infty(M,\C)$, define the integrating factor $U_{a}:SM\to\C$, unique solution to 
$$
(X+a)U_a=0\quad (SM),\qquad U_{a}|_{\p_+SM}=1,
$$
whose integral expression is given by
$$
U_{a}(x,v)=\exp\(-\int^0_{-\tau(x,-v)}a(\gamma_{x,v}(s))\,ds\),\quad (x,v)\in SM.
$$
By solving explicitly the transport equation along the geodesic, one can show that
$$
U_{a}(\varphi_t(x,v))=\exp\({-\int_0^t a(\gamma_{x,v}(s))\,ds}\),\quad (x,v)\in SM,
$$
and hence the following integral formula holds:
$$
I_{a}f(x,v)=\int_0^{\tau(x,v)}U_{a}^{-1}(\varphi_t(x,v))f(\varphi_t(x,v))\,dt,\quad (x,v)\in\p_+SM. 
$$
With the $U_a$ notation, notice that the function $w^\sharp$ defined in Section \ref{sec:main} is nothing but $w^\sharp(x,v)=U_{a}(x,v)w_\psi(x,v)$, and $Q_a$ defined in \eqref{eq:Qa} takes the expression
$$
Q_{a}w(x,v):=\begin{cases}
w(x,v)\quad&(x,v)\in\p_+SM,\\
U_{a}(x,v)(w\circ\alpha)(x,v)\quad&(x,v)\in\p_-SM.
\end{cases}
$$
The space of those $w$ for which $w^\sharp$ is smooth in $SM$ is denoted by
\begin{align*}
    \cS_{a}^\infty(\p_+SM,\C) &:= \{w\in C^\infty(\p_+ SM,\C):w^\sharp\in C^\infty(SM,\C)\} \\
    &= \{w\in C^\infty(\p_+ SM,\C):Q_{a}w\in C^\infty(\partial(SM),\C)\}, 
\end{align*}
where the second equality is a characterization in terms of the operator $Q_{a}$, proved in \cite[Lemma~5.1]{PSU1}.

Another characterization of the $B_a$ operator defined in \eqref{eq:Ba} is that, for any smooth function $\psi:SM\to \C$, we have
\begin{align*}
I_{a}((X+a)\psi)(x,v)&=\exp\({\int_0^{\tau(x,v)}a(\gamma_{x,v}(t))\,dt}\)\psi\circ\alpha(x,v)-\psi(x,v)\\
&=B_{a}\psi\big|_{\p SM}(x,v).
\end{align*}

\subsection{Geometry and Fourier analysis on \texorpdfstring{$SM$}{SM}}\label{3.2}
Since $M$ is oriented there is a circle action on the fibres of $SM$ with infinitesimal generator $V$ called the vertical vector field. We complete $X,V$ to a global frame of $T(SM)$ by defining the vector field $X_\perp:=[X,V]$, where $[\cdot,\cdot]$ is the Lie bracket for vector fields. For any two functions $u,v:SM\to\C$ define an inner product:
$$
\langle u,v\rangle_{L^2(SM)}=\int_{SM} u\overline v\,d\Sigma^3,
$$
where $d\Sigma^3$ is the Liouville measure of $g$ on $SM$. The space $L^2(SM,\C)$ decomposes orthogonally as a direct sum
$$
L^2(SM,\C)=\bigoplus_{k\in\mathbb Z}H_k
$$
where $H_k$ is the eigenspace of $-iV$ corresponding to the eigenvalue $k$. Any function $u\in C^{\infty}(SM,\C)$ has a Fourier series expansion
$$
u=\sum_{k=-\infty}^{\infty}u_k, \qquad u_k\in \Omega_k:=C^{\infty}(SM,\C)\cap H_k.
$$
We recall the first order elliptic operators due to Guillemin and Kazhdan \cite{GK}, defined by $\eta_\pm = \frac{1}{2} (X \pm iX_\perp)$. By the commutation relations $[-iV,\eta_+]=\eta_+$ and $[-iV,\eta_-]=-\eta_-$ we see that
\[ \eta_+:\Omega_k\to\Omega_{k+1},\qquad\eta_-:\Omega_k\to\Omega_{k-1}. \]
For the sequel, let us denote, for any $k\in \Z$, $\ker^k \eta_{\pm} := \Omega_k \cap \ker \eta_\pm$. 

An important tool in our approach is the {\it fiberwise Hilbert transform} $H:C^\infty(SM,\C)\to C^\infty(SM,\C)$, which we define in terms of Fourier coefficients as
\begin{align*}
    H(u_k)=-i\sgn(k)\ u_k, \qquad (\text{with the convention } \sgn(0)=0).
\end{align*}

The following commutator formula, which was derived by Pestov and Uhlmann in \cite{PeU1} and generalized in \cite{PSU1}, will play an important role.
\begin{equation}\label{[H,X]}
[ H,X+a]u=X_\perp u_0+\(X_\perp u\)_0,\quad u\in C^\infty(SM,\C).
\end{equation}
This formula has been frequently used in recent works on inverse problems, see \cite{PSU3,PSU2,PSU1,PeU2,PeU1,SaU}.

\subsection{The space of pairs} \label{sec:pairs}

The inner product in the space $\mathcal L^2(M,\C)$ is given by
\begin{equation}\label{inner product for pairs}
\([\alpha,f]\,\big|\,[\beta,h]\)_{\mathcal L^2(M,\C)}=\int_M \<\alpha,\overline{\beta}\>_g\,d\Vol_g+\int_M f\overline{h}\,d\Vol_g.
\end{equation}
Assume that $a\in C^\infty(M,\C)$. Consider the following operators $d_a:H^1(M,\C)\to \mathcal L^2(M,\C)$ and $\delta_a: \mathcal H^1(M,\C)\to L^2(M,\C)$ defined by
$$
d_a h=[dh,ah],\quad \delta_a[\alpha,f]=\delta\alpha-\overline af.
$$
The following integration by parts formula holds for these operators:
$$
\(\delta_a[\alpha,f]\,\big|\,h\)_{L^2(M,\C)}+\([\alpha,f]\,\big|\,d_a h\)_{\mathcal L^2(M,\C)}=(i_\nu \alpha|h)_{L^2(\p M)},
$$
where $\nu$ is the outward unit normal on $\p M$. In particular, we obtain $d_a^*=-\delta_a$. Introducing the spaces of {\it\bfseries $a$-solenoidal} and {\it\bfseries $a$-potential} pairs
\begin{align*}
\mathcal L_{a,{\rm sol}}^2(M,\C)&=\{[\alpha,f]\in \mathcal L^2(M,\C):\delta_a[\alpha,f]=0\},\\
\mathcal L_{a,{\rm pot}}^2(M,\C)&=\{d_a h:h\in H^1_0(M,\C)\},
\end{align*}
Proposition \ref{decomposition of pairs} below implies the $\cL^2$-orthogonal decompositions
\begin{align}
    \mathcal L^2(M,\C) &= \mathcal L_{a,{\rm sol}}^2(M,\C)\oplus \mathcal L_{a,{\rm pot}}^2(M,\C), \nonumber\\
    \mathcal C^\infty(M,\C) &=\mathcal C_{a,{\rm sol}}^\infty(M,\C)\oplus \mathcal C_{a,{\rm pot}}^\infty(M,\C), \label{decomp of the space of smooth pairs}
\end{align}
where we have defined $\mathcal C_{a,{\rm sol/pot}}^\infty(M,\C):=\mathcal L_{a,{\rm sol/pot}}^2(M,\C)\cap \mathcal C^\infty(M,\C)$.

\begin{Proposition}\label{decomposition of pairs}
    Let $a\in C^\infty(M,\C)$ and let $k\ge 0$ be an integer. For a given $[\alpha,f]\in\mathcal H^k(M,\C)$ there are unique $[\beta,h]\in \mathcal H^k(M,\C)$ and $b\in H^{k+1}(M,\C)\cap H^{1}_0(M,\C)$ such that $[\alpha,f]=[\beta,h]+d_a b$ and $\delta_a[\beta,h]=0$. Moreover, if $[\alpha,f]\in\mathcal C^\infty(M,\C)$ then $[\beta,h]\in \mathcal C_{a,{\rm sol}}^\infty(M,\C)$ and $b\in C^\infty(M,\C)$ with $b|_{\p M}=0$.
\end{Proposition}

\begin{proof}
    For $[\alpha,f] \in \cL^2$, consider the problem for $b\in H^1_0(M,\C)$
    \begin{align*}
	-\delta_a d_a b = -\delta_a [\alpha,f] \in H^{-1}(M,\C),\qquad b|_{\p M}=0,	
    \end{align*}
    whose weak formulation consists in finding $b\in H^1_0(M,\C)$ such that, 
    \[ (d_a b, d_a b')_{\cL^2(M,\C)} = \<  -\delta_a [\alpha,f] ,b'\>_{H^{-1},H^1_0}, \qquad \forall\ b' \in H^1_0(M,\C), \]
    where the sesquilinear form on the left-hand side, given by 
    \[ (d_a b, d_a b')_{\cL^2(M,\C)} = \int_M \<d_a b,\overline{d_a b'}\>_g\,d\Vol_g + \int_M |a|^2 b \overline{b'}\,d\Vol_g, \]
    is hermitian, continuous and coercive (since, when $b=b'$, the second term is nonnegative and the first term controls the $H^1_0$ norm by virtue of Poincar\'e's inequality). The existence and uniqueness of such a $b$ is then provided by Lax-Milgram's theorem, see e.g. \cite[Theorem 1, Sec. 6.2.1]{evans}.
    Once $b$ is constructed, set $[\beta,h] = [\alpha,f] - d_a b$ and the $\cL^2$ decomposition follows. Moreover, following results on higher order regularity for solutions of strongly elliptic equations (see for example \cite[Proposition~11.10]{Tay}), if $[\alpha,f]\in \cH^k$, then $b\in H^{k+1}\cap H^1_0$ and thus $[\beta,h]\in \cH^k$. In particular, if $[\alpha,f]$ are smooth, so are $b$ and $[\beta,h]$. 
\end{proof}

\subsection{Extension operators for $a$-solenoidal pairs}
Our aim in this subsection is to extend $a$-solenoidal pair to a larger manifold as compactly supported $a$-solenoidal pair in the $\mathcal C^\infty$ setting. We will follow the arguments of \cite{KMPT} and \cite{PaternainZhou2015}.

Here and in what follows, $\mathcal H^1_{U,a,{\rm sol}}(\widetilde M^{\rm int},\C)$ and $\mathcal C^\infty_{U,a,{\rm sol}}(\widetilde M^{\rm int},\C)$ denote the subspaces of $\mathcal H^1_{a,{\rm sol}}(\widetilde M^{\rm int},\C)$ and $\mathcal C^\infty_{a,{\rm sol}}(\widetilde M^{\rm int},\C)$, respectively, consisting of elements supported in $U$.

We start with the following lemma on the existence of $a$-solenoidal extensions that might not be compactly supported.

\begin{Lemma}[Smooth $a$-solenoidal extensions]\label{smooth solenoidal extension with non-compact support}
Let $M$ be a compact simply connected manifold contained in the interior of some Riemannian manifold $(\widetilde M,g)$ and let $a\in C^\infty(\widetilde M,\C)$. There is an open neighborhood $U$ of $M$ and a linear operator $\mathcal E_{a,U}:\mathcal C^\infty_{a,{\rm sol}}(M,\C)\to \mathcal C^\infty_{a,{\rm sol}}(U,\C)$ with $\mathcal E_{a,U}=\id$ on $M$ and $\|\mathcal E_{a,U}[\alpha,f]\|_{\mathcal H^1(U,\C)}\le C\|[\alpha,f]\|_{\mathcal H^1(M,\C)}$.
\end{Lemma}
\begin{proof}
We cover $\p M$ in $\widetilde M$ by charts $\{(\mathcal O_\kappa,\Theta_\kappa)\}_\kappa$ with semi-geodesic local coordinates, i.e. each coordinate map $\Theta_\kappa:\mathcal O_\kappa\to\R^n$ is of the form $\Theta_\kappa(p)=(x^1,\dots,x^{n-1},x^n)=(x',x^n)$ such that $\Theta^{-1}_\kappa(\{x^n=0\})\cap \mathcal O_\kappa\subset \p M$, $\Theta^{-1}_\kappa(\{x^n<0\})\cap \mathcal O_\kappa\subset M^{\rm int}$ and $(\Theta^{-1}_\kappa)_*\p_{n}=\nu$ is the unit outward (from $M$) normal to $\p M$. In these coordinates, we have
$$
g^{kn}=\delta^{kn},\quad \Gamma^n_{kn}=\Gamma^k_{nn}=0,\quad k=1,\dots,n.
$$
We determine $U\setminus M$ as the sufficiently small semi-geodesic neighborhood of $\p M$ in $\widetilde M$ such that $\overline U\subset\cup_\kappa \mathcal O_\kappa$.

Given $[\alpha,f]\in\mathcal C^\infty_{a,{\rm sol}}(M,\C)$. We extend the function $f$ and the components $\alpha_{i'}$, $i'=1,\dots,n-1$, smoothly to $U$, and denote the extensions by $h$ and $\beta_{i'}$, respectively. By the results in \cite{Seeley}, these extensions can be done in a stable way
\begin{equation}\label{H^1 norm estimates for h and beta' components}
\|h\|_{H^1(U,\C)}\le C\|f\|_{H^1(M,\C)},\quad \|\beta_{i'}\|_{H^1(U,\C)}\le C\|\alpha_{i'}\|_{H^1(M,\C)},\quad {i'}=1,\dots,n-1.
\end{equation}
Now we construct the last component $\beta_n$ in $\Theta^{-1}_\kappa(\{x^n>0\})\cap \mathcal O_\kappa$. Since we want $a$-solenoidal extension, writing $h^\kappa=h\circ\Theta^{-1}_\kappa$ and $\beta_i^\kappa=\beta_i\circ\Theta^{-1}_\kappa$, $i=1,\dots,n$, we have
\begin{equation}\label{key ode for a-solenoidal extension}
\p_n\beta_n^\kappa-\sum_{j,k<n}g^{jk}\Gamma_{jk}^n\beta_n^\kappa=\overline{a}h^\kappa
-\sum_{j,k<n}g^{jk}\p_j\beta_k^\kappa+\sum_{j,k,l<n}g^{jk}\Gamma_{jk}^l\beta_l^\kappa.
\end{equation}
Observe that the right side is known, so it is a first order linear ordinary differential equation. Given the initial condition $\beta_n^\kappa(x',0)=\alpha_n\circ\Theta^{-1}_\kappa(x',0)$, there is a unique solution $\beta_n^\kappa(x',x^n)$. In this way we construct continuous $\beta_n^\kappa$ in $\{x^n>0\}\cap \Theta_\kappa(\mathcal O_\kappa)$ which depends smoothly on $x'$. If $U$ is sufficiently close to $M$, one can show that in each chart $\mathcal O_\kappa$ the following holds
$$
|\beta_n^\kappa(x',x^n)|^2\le C_{\mathcal O_\kappa}\Big(|\alpha_n\circ\Theta^{-1}_\kappa(x',0)|^2+|h^\kappa(x',x^n)|^2+\sum_{j,k<n}|\p_j \beta_k^\kappa(x',x^n)|^2\Big),
$$
for all $(x',x^n)\in \{x^n>0\}\cap \Theta_\kappa(\mathcal O_\kappa\cap U)$. Integrating over $U$ and using compactness of $M$, we obtain
$$
\|\beta_n\|_{L^2(U,\C)}\le C\Big(\|\alpha_n\|_{L^2(\p M,\C)}+\|f\|_{H^1(M,\C)}+\sum_{j<n}\|\alpha_j\|_{H^1(M,\C)}\Big).
$$
Let $V$ be a neighborhood of $\p M$ in $M$. Then for all $(x',x^n)\in \{x^n<0\}\cap\Theta_\kappa(\mathcal O_\kappa\cap V)$ we can show that
$$
|\alpha_n(x',0)|^2\le C_{\mathcal O_\kappa}\left( |\alpha_n(x',x^n)|^2+\int^0_{x^n}|\alpha_n(x',x^n)|^2\,dx^n+\int^0_{x^n}|\p_n\alpha_n(x',x^n)|^2\,dx^n\right).
$$
Integrating over $M$ and using compactness of $M$, we obtain
$$
\|\alpha_n\|_{L^2(\p M,\C)}\le C\|\alpha\|_{H^1(M,\C)},
$$
and hence
$$
\|\beta_n\|_{L^2(U,\C)}\le C\Big(\|f\|_{H^1(M,\C)}+\|\alpha\|_{H^1(M,\C)}\Big).
$$
Therefore, combining the latter inequality together with \eqref{H^1 norm estimates for h and beta' components} and \eqref{key ode for a-solenoidal extension}, we come to
$$
\|[\beta,h]\|_{\mathcal H^1(U,\C)}\le C\|[\alpha,f]\|_{\mathcal H^1(M,\C)}.
$$
Now, we want to show the smoothness of $\beta_n$. Differentiating \eqref{key ode for a-solenoidal extension} with respect to $x^n$, we show that $\beta_n^\kappa(x',x^n)$ is smooth in $\{x^n\ge 0\}\cap\Theta_\kappa(\mathcal O_\kappa)$. Moreover, using \eqref{key ode for a-solenoidal extension} and induction on the order of derivative with respect to $x^n$, we show
$$
\p_n^m\beta_n^\kappa(x',0)=\p_n^m\alpha_n\circ\Theta_\kappa^{-1}(x',0)
$$
for all $m\ge 0$ integers. In other words, $\p_n^m\beta_n(x',0)$ agree with $\p_n^m\alpha_n\circ\Theta^{-1}_\kappa(x',0)$. Therefore, we obtain a smooth $a$-solenoidal pair
$$
\mathcal E_{a,U}[\alpha,f]=\begin{cases}
[\alpha,f]&\text{on}\quad M,\\
[\beta,h]&\text{on}\quad U\setminus M,
\end{cases}
$$
with $\|\mathcal E_{a,U}[\alpha,f]\|_{\mathcal H^1(U,\C)}\le C\|[\alpha,f]\|_{\mathcal H^1(M,\C)}$.
\end{proof}

\begin{Proposition}\label{extension of a-solenoidal pairs with compact support}
Let $M$ be a compact simply connected manifold contained in the interior of some Riemannian manifold $(\widetilde M,g)$ and let $a\in C^\infty(\widetilde M,\C)$. There is a precompact neighborhood $W$ of $M$ in $\widetilde M^{\rm int}$ and a bounded map $\mathcal E_a:\mathcal H^1_{a,{\rm sol}}(M,\C)\to \mathcal H^1_{W,a,{\rm sol}}(\widetilde M^{\rm int},\C)$ such that $\mathcal E_a|_{M}=\id$. Moreover, the restriction of $\mathcal E_a$ to $\mathcal C^\infty_{a,{\rm sol}}(M,\C)$ maps $\mathcal C^\infty_{a,{\rm sol}}(M,\C)$ into $\mathcal C^\infty_{W,a,{\rm sol}}(\widetilde M^{\rm int},\C)$.
\end{Proposition}

\begin{proof}
Given $[\alpha,f]\in\mathcal C^\infty_{a,{\rm sol}}(M,\C)$, Lemma~\ref{smooth solenoidal extension with non-compact support} implies the existence of a neighborhood $U$ of $M$ and a linear operator $\mathcal E_{a,U}:\mathcal C^\infty_{a,{\rm sol}}(M,\C)\to \mathcal C^\infty_{a,{\rm sol}}(U,\C)$ with $\mathcal E_{a,U}=\id$ on $M$ and
$$
\|\mathcal E_{a,U}[\alpha,f]\|_{\mathcal H^1(U,\C)}\le C\|[\alpha,f]\|_{\mathcal H^1(M,\C)}.
$$
Consider an open precompact set $W$ such that $U\subset\overline U\subset W\subset\overline W\subset \widetilde M^{\rm int}$. Then, using \cite{Seeley}, we extend $\mathcal E_{a,U}[\alpha,f]$ to $\widetilde M^{\rm int}$ and multiply the extension by a smooth cut-off function which is equal to $1$ in $U$ and supported in $W$, and we denote the resultant pair by $[\beta,h]$. Combining the result of \cite{Seeley} and Lemma~\ref{smooth solenoidal extension with non-compact support} implies that
$$
\|[\beta,h]\|_{\mathcal H^1(\widetilde M^{\rm int},\C)}\le C\|[\alpha,f]\|_{\mathcal H^1(M,\C)}.
$$
Set $w=\delta_a[\beta,h]$ and $D=W\setminus\overline M$. We have $\supp w\subset W\setminus U\subset D$.

We claim that $(w|v)_{L^2(D,\C)}=0$ for all $v\in H^1(D,\C)\cap\ker d_a$. Then, by \cite[Corollary~1.6]{Delay}  (see also \cite[Section~5.1]{Qui}), there is a smooth one-form $\beta_D$ on $\widetilde M^{\rm int}$ such that $\delta \beta_D=-w$ and $\supp\beta_D\subset W\setminus M^{\rm int}$. Moreover, $\beta_D$ satisfies
$$
\|\beta_D\|_{H^1(\widetilde M^{\rm int},\C)}\le C\|w\|_{L^2(\widetilde M^{\rm int},\C)}.
$$
We define
$$
\mathcal E_a[\alpha,f]=[\beta+\beta_D,h].
$$
Then $\mathcal E_a[\alpha,f]|_{M}=[\alpha,f]$, $\supp\mathcal E_a[\alpha,f]\subset W$ and $\delta_a\mathcal E_a[\alpha,f]=\delta_a[\beta,h]+\delta\beta_D=w-w=0$. Hence $\mathcal E_a[\alpha,f]\in\mathcal C^\infty_{W,a,{\rm sol}}(\widetilde M^{\rm int},\C)$ and
$$
\|\mathcal E_{a}[\alpha,f]\|_{\mathcal H^1(\widetilde M^{\rm int},\C)}\le C\|[\alpha,f]\|_{\mathcal H^1(M,\C)}.
$$
Since $\mathcal C^\infty(\widetilde M^{\rm int},\C)$ is dense in $\mathcal H^1(\widetilde M^{\rm int},\C)$ under the $\mathcal H^1$-norm, we extend $\mathcal E_a$ to a bounded map  $\mathcal E_a:\mathcal H^1_{a,{\rm sol}}(M,\C)\to \mathcal H^1_{W,a,{\rm sol}}(\widetilde M^{\rm int},\C)$ with $\mathcal E_a|_{M}=\id$.

Now it is left to prove that $(w|v)_{L^2(D,\C)}=0$ for all $v\in H^1(D,\C)\cap\ker d_a$. For this, we study the solutions of the homogeneous problem. Let $v\in H^1(D,\C)$ be a solution of
$$
\begin{cases}
&(-\Delta_g+|a|^2)v=0\text{ in }D,\\
&\p_N v=0\text{ on }\p D,
\end{cases}
$$
where $N$ is the unit outward normal on $\p D$. Then by Green's formula
$$
\|\nabla v\|_{L^2(D,\C)}^2+\|a v\|_{L^2(D,\C)}^2=((-\Delta_g+|a|^2)v|v)_{L^2(D,\C)}+(\p_N v|v)_{L^2(\p D,\C)}=0.
$$
Thus, $\nabla v\equiv0$ and $av\equiv0$. In other words, $v\in \ker d_a$. Let $K_a$ denotes the set of the solutions of the homogeneous problem, then
$$
K_a=\{v\in H^1(D,\C):\nabla v\equiv0,\,av\equiv0\}=H^1(D,\C)\cap\ker d_a.
$$
If $a\equiv 0$, then $K_a$ consists of constant solutions. Hence, for $v=\const\in K_a$, integration by parts gives
\begin{align*}
(w|v)_{L^2(D,\C)}&=(\delta_a[\beta,h]|v)_{L^2(D,\C)}=-(i_\nu\alpha|v)_{L^2(\p M,\C)}\\
&=-(\delta_a[\alpha,f]|v)_{L^2(M,\C)}-([\alpha,f]|d_a v)_{\mathcal L^2(M,\C)}=0.
\end{align*}

Now, if $a\not\equiv 0$, then $K_a=\{0\}$. Hence, for $v=0\in K_a$ we trivially have $(w|v)_{L^2(D,\C)}=(w|0)_{L^2(D,\C)}=0$.
\end{proof}

\section{Surjectivity results for the adjoints}

The main purpose of this section is to obtain the surjectivity theorem \ref{main surjectivity result}, of the adjoint $\cI_a^*$, on which our range characterizations hinge. We first compute the adjoints of $I_a$ and $\cI_a$ in section \ref{sec:adjoints}, then prove Theorem \ref{main surjectivity result} in section \ref{sec:surjectivity}. 

\subsection{Adjoints of $I_{a}$ and $\cI_{a}$} \label{sec:adjoints}

Let $d\Sigma^2$ be the volume form on $\p(SM)$. Denote by $L_\mu^2(\p_+ SM,\C)$ the completion of $C_c^\infty(\partial_+ SM,\C)$ for the inner product
\begin{align*}
    \<h,h'\>_{L^2_\mu(\p_+SM,\C)}=\int_{\p_+ SM}h\overline{h'}\, \mu\ d\Sigma^2, \qquad \mu(x,v):=\< v,\nu_x\>_{g(x)}.   
\end{align*}
As in \cite{PSU2}, using the integral representation for $I_{a}$ and Santal\'o formula \cite[Lemma~A.8]{DPSU}, one can show that $I_{a}$ can be extended to a bounded operator $I_{a}:L^2(SM,\C)\to L_\mu^2(\p_+ SM,\C)$.

Consider the dual $I_{a}^*:L^2_\mu(\p_+SM,\C)\to L^2(SM,\C)$ of $I_{a}$, for which we now find an expression. For this consider $h\in L^2_\mu(\p_+SM,\C)$. Using Santal\'o's formula, for $f\in L^2(SM,\C)$, we compute
\begin{equation}\label{dual of I_a}
\begin{aligned}
\<I_{a} f,h\>&_{L^2_\mu(\p_+SM,\C)}\\
&=\int_{\p_+SM}\(\int_0^{\tau(x,v)}U_{a}^{-1}(\varphi_t(x,v)) f(\varphi_t(x,v))\,dt\)\overline{h(x,v)}\,\mu\ d\Sigma^2 \\
&=\int_{\p_+SM}\(\int_0^{\tau(x,v)}U_{a}^{-1}(\varphi_t(x,v)) f(\varphi_t(x,v))\overline{h_\psi(\varphi_t(x,v))}\,dt\)\,\mu\ d\Sigma^2\\
&=\int_{\p_+SM}\(\int_0^{\tau(x,v)} f(\varphi_t(x,v))\overline{\overline{U_{a}^{-1}(\varphi_t(x,v))}h_\psi(\varphi_t(x,v))}\,dt\)\,\mu\ d\Sigma^2\\
&=\int_{SM}\(f(x,v)\,\overline{\overline{U_{a}^{-1}(x,v)}h_\psi(x,v)}\)\,d\Sigma^{3}(x,v).
\end{aligned}
\end{equation}
Therefore, we obtain
\begin{equation}\label{general formula for the adjoint}
I_{a}^* h=\overline{U_{a}^{-1}}h_\psi=U_{-\overline a}h_\psi.
\end{equation}
Moreover, if $\imath_k:H_k\to L^2(SM,\C)$ denotes the usual inclusion map, the one can show that
$$
I_{m,a}^*(h)=(\overline{U_{a}^{-1}}h_\psi)_m=(U_{-\overline a}h_\psi)_m,
$$
where $I_{m,a}:=I_{a}\circ \imath_k$.

From \eqref{dual of I_a}, one also can get the following explicit expressions for the adjoints of $I_{a}^0:L^2(M,\C)\to L^2_\mu(\p_+SM,\C)$ and $I^1_{a}:L^2(\Lambda^1(M),\C)\to L^2_\mu(\p_+SM,\C)$:
\begin{equation}\label{adjoints of I^0 and I^1}
\begin{aligned}
(I_{a}^0)^*(h)(x)&=\int_{S_x M}U_{-\overline a}(x,v)\,h_\psi(x,v)\,d\sigma_x(v),\\
(I_{a}^1)^*(h)^i(x)&=\(\int_{S_x M}v^i U_{-\overline a}(x,v)\, h_\psi(x,v)\,d\sigma_x(v)\),
\end{aligned}
\end{equation}
where $d\sigma_x$ is the measure on $S_xM$. The identifications of $H_0$ with $L^2(M,\C)$ and $H_{-1}\oplus H_1$ with $L^2(\Lambda^1(M),\C)$ imply that
\begin{align}
(I_{a}^0)^*(h)&=2\pi (U_{-\overline a}h_\psi)_{0},\label{adjoint of I^0}\\
(I_{a}^1)^*(h)&=\pi(U_{-\overline a}h_\psi)_{-1}+\pi (U_{-\overline a}h_\psi)_{1},\label{adjoint of I^1}
\end{align}
see \cite[Remark~5.2]{PSU3} for more details. Then we have
$$
\cI_a^*h=[(I_a^1)^*(h),(I_a^0)^*(h)].
$$
\begin{Remark}\label{image of I^* is in potential space}
    Note that $\Im \cI^*_a$ is in the orthogonal complement to $\ker \cI_a$, and hence if $\ker \cI_a= \cL^2_{a,{\rm pot}}(M,\C)$ then
    \begin{align*}
	\Im \cI^*_a\subset \cL^2_{a,{\rm sol}}(M,\C).	
    \end{align*}
\end{Remark}

\subsection{Surjectivity of $\cI^*_{a}$} \label{sec:surjectivity}

The aim of this section is to prove the following result which is the analogue of the corresponding results in \cite{AinAs,DaU,PSU3,PeU1,PeU2}.

\begin{Theorem}\label{main surjectivity result}
Let $(M,g)$ be a simple surface and let $a\in C^\infty(M,\C)$. Suppose that $[\alpha,f]\in\mathcal C_{a,\rm sol}^\infty(M,\C)$. Then there is $w\in \mathcal S_{a}^\infty(\p_+SM,\C)$ such that $\mathcal I_{a}^*(w)=[\alpha,f]$.
\end{Theorem}
Applying Theorem \ref{main surjectivity result} to a pair of the form $[\alpha,0]$ yields the following

\begin{Corollary}\label{main surjectivity result for solenoidal form}
Let $(M,g)$ be a simple surface and let $a\in C^\infty(M,\C)$. Suppose that $\alpha:TM\to \C$ is a smooth, solenoidal one-form. Then there is $w\in \mathcal S_{a}^\infty(\p_+SM,\C)$ such that $(I_{a}^1)^*(w)=\alpha$ and $(I_{a}^0)^*(w)=0$.
\end{Corollary}

\begin{proof}[Proof of Theorem~\ref{main surjectivity result}]
We embed $M$ into the interior of a compact surface $\widetilde M$ with boundary and extend the metric $g$ to $\widetilde M$ and keep the same notation for the extension, choosing $(\widetilde M,g)$ to be sufficiently close to $(M,g)$ so that it remains simple. We also extend the attenuation coefficient $a$ to $\widetilde M$ smoothly and to be real valued, and keep the same notation for the extensions.

Before proceeding further let us introduce more conventions which will be used in this section. If $A$ is a notation for some object in the context of the surface $(M,g)$, then by $\widetilde A$ we denote the same object but in the context of the extended surface $(\widetilde M,g)$. For example, the notation $\tilde{\mathcal I}_a$ will denote the ray transform with domain $\mathcal L^2(\widetilde M,\C)$ and $\widetilde{\mathcal N}_{a}:=(\tilde{\mathcal I}_{a})^*\tilde{\mathcal I}_{a}$.

According to Remark~\ref{image of I^* is in potential space}, we have
\begin{equation}\label{delta_a N_a=0}
\delta_a\, \widetilde{\mathcal N}_{a}=0.
\end{equation}

Let $r_M$ denote the restriction operator from $\widetilde M$ to $M$. We show the following:
\begin{Lemma}\label{normal operator is surjective}
The operator
\begin{equation}\label{r_MN}
r_M \widetilde{\mathcal N}_{a}:\mathcal C_{0}^\infty(\widetilde M^{\rm int},\C)\to \mathcal C_{a,\rm sol}^\infty(M,\C)
\end{equation}
is surjective.
\end{Lemma}

Assuming this result, we give the proof of Theorem~\ref{main surjectivity result}. Suppose that $[\alpha,f]\in \mathcal C_{a,\rm sol}^\infty(M,\C)$. Then Lemma~\ref{normal operator is surjective} ensure the existence of $[\beta,h]\in\mathcal C_{0}^\infty(\widetilde M^{\rm int},\C)$ such that
$$
[\alpha,f]=r_M \widetilde{\mathcal N}_a[\beta,h]=r_M(\tilde{\mathcal I}_a)^*\tilde{\mathcal I}_a[\beta,h].
$$
Recall that $\widetilde U_{a}$ is the unique solution to the initial value problem
$$
(X+a)\widetilde U_{a}=0\text{ in }S\widetilde M,\quad \widetilde U_{a}|_{\p_+S\widetilde M}=1.
$$
The integral expression for $\widetilde U_{a}$ is as follows:
$$
\widetilde U_{a}(x,v)=\exp\(-\int^0_{-\tilde\tau(x,-v)}a(\gamma_{x,v}(s))\,ds\),\quad (x,v)\in S\widetilde M.
$$
Here $\tilde\tau(x,v)$ is the unique non-negative time when the geodesic $\gamma_{x,v}$, with $\gamma_{x,v}(0)=x$ and $\dot\gamma_{x,v}(0)=v$, exits $\widetilde M$. Let us now define
$$
\widetilde w(x,v):=\int_{-\tilde\tau(x,-v)}^{\tilde\tau(x,v)}\widetilde U_{a}^{-1}(\gamma_{x,v}(t),\dot\gamma_{x,v}(t))\left\{\beta_j(\gamma_{x,v}(t))\,\dot\gamma_{x,v}^j(t)+h(\gamma_{x,v}(t))\right\}\,dt.
$$
Note that $\widetilde w\in C^\infty(S\widetilde M^{\rm int},\C)$. From the definition, one can show that
$$
\tilde{\mathcal I}_{a}[\beta,h]=\widetilde w|_{\p_+S\widetilde M}.
$$
Using the formulas \eqref{adjoint of I^0} and \eqref{adjoint of I^1} for the adjoints and using that $\tilde{\mathcal I}_{a}=\tilde I_{a}^1+\tilde I_{a}^0$, we can obtain
\begin{align*}
[\alpha,f]&=r_M(\tilde{\mathcal I}_{a})^*\tilde{\mathcal I}_{a}[\beta,h]\\
&=\(\pi(U_{-\overline a}\widetilde w)_{-1}+2\pi(U_{-\overline a}\widetilde w)_0+\pi(U_{-\overline a}\widetilde w)_1\)\big|_{SM}\\
&=({\mathcal I}_{a})^*(\widetilde U_{-\overline a}\widetilde w)|_{\p_+SM}.
\end{align*}
In the last step we used the fact that $\widetilde U_{-\overline a}$ and $U_{-\overline a}$ are related by
$$
\widetilde U_{-\overline a}(x,v)=U_{-\overline a}(x,v)(\widetilde U_{-\overline a}|_{\p_+SM})_\psi(x,v)\text{ for all }(x,v)\in SM.
$$
This is easy to see from the integral expressions for $\widetilde U_{-\overline a}$ and $U_{-\overline a}$. Setting $w=(\widetilde U_{-\overline a}\widetilde w)|_{\p_+SM}$ we finish the proof.
\end{proof}

To prove Lemma~\ref{normal operator is surjective} we need the following result.

\begin{Lemma}\label{normal operator is elliptic PDO}
The normal operator $\widetilde{\mathcal N}_{a}$ is a pseudodifferential operator of order $-1$ in $\widetilde M^{\rm int}$. Moreover, $\widetilde{\mathcal N}_{a}$ is elliptic on $a$-solenoidal pairs.
\end{Lemma}

We say that $\widetilde{\mathcal N}_{a}$ is \emph{elliptic on $a$-solenoidal pairs}, if $\diag(d_a\Lambda \delta_a,\widetilde{\mathcal N}_{a})$, acting on pairs, is elliptic (as a system of pseudodifferential operators of order $-1$), where $\Lambda$ is a proper pseudodifferential operator on $\widetilde M^{\rm int}$ with principal symbol $1/|\xi|^3$. Recall that $\diag(d_a\Lambda \delta_a,\widetilde{\mathcal N}_{a})$ is an elliptic system if $\det\sigma_p(\diag(d_a\Lambda \delta_a,\widetilde{\mathcal N}_{a}))(x,\xi)\neq 0$ for $(x,\xi)\in T\widetilde M\setminus \{0\}$; see \cite[page~46]{Shubin}.

\begin{proof}
First, we prove that $\widetilde{\mathcal N}_{a}$ is a pseudodifferential operator of order $-1$ in $\widetilde M^{\rm int}$. Recall that by $\widetilde U_{a}$ we denote the unique solution to
$$
(X+a)\widetilde U_{a}=0\quad (S\widetilde M),\qquad \widetilde U_{a}|_{\p_+S\widetilde M}=1.
$$
Recall also that the normal operator is as follows $\widetilde{\mathcal N}_{a}:\mathcal L^2(\widetilde M,\C)\to\mathcal L^2(\widetilde M,\C)$. Therefore, we introduce the following notation
$$
\widetilde{\mathcal N}_{a}[\alpha,f]=[\,\widetilde{\mathcal N}_{a}^{11}\alpha+\widetilde{\mathcal N}_{a}^{10}f\,,\,\widetilde{\mathcal N}_{a}^{01}\alpha+\widetilde{\mathcal N}_{a}^{00}f],
$$
with
$$
\widetilde{\mathcal N}_{a}^{11}:=(\tilde I_{a}^1)^*\tilde I_{a}^1,\quad \widetilde{\mathcal N}_{a}^{10}:=(\tilde I_{a}^1)^*\tilde I_{a}^0,\quad \widetilde{\mathcal N}_{a}^{01}:=(\tilde I_{a}^0)^*\tilde I_{a}^1,\quad \widetilde{\mathcal N}_{a}^{00}:=(\tilde I_{a}^0)^*\tilde I_{a}^0.
$$
Using \eqref{adjoints of I^0 and I^1}, one can show that
\begin{align*}
\(\widetilde{\mathcal N}_{a}^{11}\alpha\){}^{i'}(x)&=\int_{S_x\widetilde M}v^{i'}\widetilde U_{-\overline a}(x,v)\int_{-\tilde \tau(x,-v)}^{\tilde \tau(x,v)}\widetilde U_{a}^{-1}(\varphi_{x,v}(t))\alpha_i(\gamma_{x,v}(t))\dot\gamma_{x,v}^i(t)\,dt\,d\sigma_x(v),\\
\(\widetilde{\mathcal N}_{a}^{10}f\){}^{i'}(x)&=\int_{S_x\widetilde M}v^{i'}\widetilde U_{-\overline a}(x,v)\int_{-\tilde \tau(x,-v)}^{\tilde \tau(x,v)}\widetilde U_{a}^{-1}(\varphi_{x,v}(t))f(\gamma_{x,v}(t))\,dt\,d\sigma_x(v),\\
\(\widetilde{\mathcal N}_a^{01}\alpha\)(x)&=\int_{S_x\widetilde M}\widetilde U_{-\overline a}(x,v)\int_{-\tilde \tau(x,-v)}^{\tilde \tau(x,v)}\widetilde U_{a}^{-1}(\varphi_{x,v}(t))\alpha_i(\gamma_{x,v}(t))\dot\gamma_{x,v}^i(t)\,dt\,d\sigma_x(v),\\
\(\widetilde{\mathcal N}_{a}^{00}f\)(x)&=\int_{S_x\widetilde M}\widetilde U_{-\overline a}(x,v)\int_{-\tilde \tau(x,-v)}^{\tilde \tau(x,v)}\widetilde U_{a}^{-1}(\varphi_{x,v}(t))f(\gamma_{x,v}(t))\,dt\,d\sigma_x(v).
\end{align*}
Following \cite{FSU,HS}, we use \cite[Lemma~B.1]{DPSU} to deduce that $\widetilde{\mathcal N}_{a}$ is a pseudodifferential operator of order $-1$, and the principal symbols of the above operators are as follows:
\begin{align*}
\sigma_p(\widetilde{\mathcal N}_{a}^{11})^{i'i}(x,\xi)&=2\pi\int_{S_x\widetilde M}\omega^{i'}\omega^{i}\delta(\<\omega,\xi\>_{g(x)})\widetilde U_{-2\Re(a)}(x,\omega)\,d\sigma_x(\omega),\\
\sigma_p(\widetilde{\mathcal N}_{a}^{10})^{i'}(x,\xi)&=2\pi\int_{S_x\widetilde M}\omega^{i'}\delta(\<\omega,\xi\>_{g(x)})\widetilde U_{-2\Re(a)}(x,\omega)\,d\sigma_x(\omega),\\
\sigma_p(\widetilde{\mathcal N}_{a}^{01})^{i}(x,\xi)&=2\pi\int_{S_x\widetilde M}\omega^i\delta(\<\omega,\xi\>_{g(x)})\widetilde U_{-2\Re(a)}(x,\omega)\,d\sigma_x(\omega),\\
\sigma_p(\widetilde{\mathcal N}_{a}^{00})(x,\xi)&=2\pi\int_{S_x\widetilde M}\delta(\<\omega,\xi\>_{g(x)})\widetilde U_{-2\Re(a)}(x,\omega)\,d\sigma_x(\omega).
\end{align*}
Now, we prove ellipticity. For this, note that the ellipticity of $\diag(d_a\Lambda \delta_a,\widetilde{\mathcal N}_{a})$ is equivalent to saying that the principal symbol $\sigma_p(\diag(d_a\Lambda \delta_a,\widetilde{\mathcal N}_{a}))(x,\xi)$, acting on pairs, is injective for every $(x,\xi)\in T\widetilde M\setminus \{0\}$; see the comments preceding \cite[Definition~7.1]{Trev73}. Assume that $\sigma_p(\widetilde{\mathcal N}_{a})[\alpha,f]=0$ and $\sigma_p(d_a\Lambda \delta_a)[\alpha,f]=0$ for some $x$ and $\xi\neq 0$. Then it follows that
\begin{equation}\label{alpha at xi is zero}
\alpha_i(x)\xi^i=0
\end{equation}
and
\begin{align*}
(\sigma_p(\widetilde{\mathcal N}_{a})[\alpha,f],[\alpha,f])_{g}&=2\pi\int_{S_x \widetilde M}|\alpha_{i}(x)\omega^i+f(x)|^2\delta(\<\omega,\xi\>_{g(x)})\widetilde U_{-2\Re(a)}(x,\omega)\,d\sigma_x(\omega)\\
&=0,
\end{align*}
where the inner product $(\cdot,\cdot)_g$ is as in \eqref{inner product for pairs} before the integration. Note that $U_{-2\Re(a)}>0$ and that the set $S_{x,\xi}:=\{\omega\in S_x\widetilde M:\<\omega,\xi\>_{g(x)}=0\}$ is non-empty. Therefore, for all such $\omega$, we get
\begin{equation}\label{sum of pair is zero for all perp to xi}
\alpha_{i}(x)\omega^i+f(x)=0.
\end{equation}
Since $-\omega$ is also in $S_{x,\xi}$, we have
$$
-\alpha_{i}(x)\omega^i+f(x)=0.
$$
These two equalities imply that $f(x)=0$. Then \eqref{alpha at xi is zero} and \eqref{sum of pair is zero for all perp to xi} show that $\alpha(x)=0$. Thus, $\widetilde{\mathcal N}_{a}$ is elliptic on $a$-solenoidal pairs.
\end{proof}

Now we give the proof of Lemma~\ref{normal operator is surjective}.

\begin{proof}[Proof of Lemma~\ref{normal operator is surjective}]
For the proof we closely follow the arguments in \cite{DaU}. First, we will show that $r_M \widetilde{\mathcal N}_{a}$ has closed range. Since $\diag(\widetilde{\mathcal N}_{a},d_a\Lambda \delta_a)$, acting on pairs, is elliptic, there is a parametrix
$$
P=\(\begin{matrix}
X&Y\\
Z&T
\end{matrix}\)
$$
such that
\begin{equation}\label{eqn::CP=id}
\diag(\widetilde{\mathcal N}_{a},d_a\Lambda \delta_a)P=\(\begin{matrix}
\widetilde{\mathcal N}_{a}X&\widetilde{\mathcal N}_{a}Y\\
d_a\Lambda \delta_a Z&d_a\Lambda \delta_a T
\end{matrix}\)\equiv \id,
\end{equation}
and
\begin{equation}\label{eqn::PC=id}
P\diag(\widetilde{\mathcal N}_{a},d_a\Lambda \delta_a)=\(\begin{matrix}
X\widetilde{\mathcal N}_{a}&Yd_a\Lambda \delta_a\\
Z\widetilde{\mathcal N}_{a}&Td_a\Lambda \delta_a
\end{matrix}\)\equiv \id,
\end{equation}
where $\equiv$ means equivalence up to a smoothing operator.

Let us use the convention that for two pairs of operators the multiplication is defined as
$$
(A,B)(C,D)=AC+BD=(A,0)(C,0)+(0,B)(0,D).
$$
If we denote $C_a:=(\widetilde{\mathcal N}_{a},d_a\Lambda \delta_a)$, then from \eqref{eqn::CP=id} and \eqref{eqn::PC=id} there is a pair of pseudodifferential operators $(A,B)$ such that
\begin{equation}\label{ABC=1}
\begin{aligned}
(A,B)C_a&=(A,0)(\widetilde{\mathcal N}_{a},0)+(0,B)(0,d_a\Lambda \delta_a)\equiv\id,\\
C_a(A,B)&=(\widetilde{\mathcal N}_{a},0)(A,0)+(0,d_a\Lambda \delta_a)(0,B)\equiv\id.
\end{aligned}
\end{equation}
In fact, $A=\frac12 X$ and $B=\frac12 T$. Using \eqref{delta_a N_a=0}, we show that
$$
-\delta_a C_a=(-\delta_a\widetilde{\mathcal N}_{a},-\delta_a d_a\Lambda \delta_a)=(0,-\delta_a d_a\Lambda \delta_a)=-\delta_a d_a(0,\Lambda \delta_a).
$$
The operator $-\delta_a d_a$ is $-\Delta_g+|a|^2$, which has a proper parametrix $(-\Delta_g+|a|^2)^{-1}$. Then
$$
(0,\Lambda \delta_a)=-(-\Delta_g+|a|^2)^{-1}\delta_a C_a.
$$
Therefore,
$$
(\widetilde{\mathcal N}_{a},0)(A,0)-d_a(-\Delta_g+|a|^2)^{-1}\delta_a C_a(0,B)\equiv\id.
$$
Since $C_a(0,B)=C_a(A,B)-(\widetilde{\mathcal N}_a,0)(A,0)\equiv \id-(\widetilde{\mathcal N}_a,0)(A,0)$ and $\delta_a \widetilde{\mathcal N}_a=0$, this imply that
$$
\widetilde{\mathcal N}_{a}A-d_a(-\Delta_g+|a|^2)^{-1}\delta_a=\id+K,
$$
where $K$ is a smoothing operator in $\widetilde M^{\rm int}$. Restricting to $\mathcal C_{0,a,{\rm sol}}^\infty(\widetilde M^{\rm int},\C)$, we obtain
$$
\widetilde{\mathcal N}_{a}A[\alpha,f]=[\alpha,f]+K[\alpha,f],\quad\text{for all}\quad [\alpha,f]\in\mathcal C_{0,a,{\rm sol}}^\infty(\widetilde M^{\rm int},\C).
$$
By Proposition~\ref{extension of a-solenoidal pairs with compact support}, there is a precompact neighborhood $W$ of $M$ in $\widetilde M^{\rm int}$ such that there is a bounded map $\mathcal E_a:\mathcal H^1_{a,{\rm sol}}(M,\C)\to \mathcal H^1_{W,a,{\rm sol}}(\widetilde M^{\rm int},\C)$ such that $\mathcal E_a|_{M}=\id$ and $\mathcal E_a(\mathcal C^\infty_{a,{\rm sol}}(M,\C))\subset \mathcal C^\infty_{W,a,{\rm sol}}(\widetilde M^{\rm int},\C)$. Then we have on $\mathcal H^1_{a,{\rm sol}}(M,\C)$
$$
r_M\widetilde{\mathcal N}_{a}A\mathcal E_a=\id+r_M K\mathcal E_a.
$$
Since $K$ is smoothing in $\widetilde M^{\rm int}$, the operator $r_M K\mathcal E_a$ is compact. Hence, the operator $\id+r_M K\mathcal E_a:\mathcal H^1_{a,{\rm sol}}(M,\C)\to \mathcal H^1_{a,{\rm sol}}(M,\C)$ has closed and finite codimensional range. Since $K$ is smoothing, we also get that the operator $\id+r_M K\mathcal E_a:\mathcal C^\infty_{a,{\rm sol}}(M,\C)\to \mathcal C^\infty_{a,{\rm sol}}(M,\C)$ also has closed and finite codimensional range. Therefore, $r_M\widetilde{\mathcal N}_{a}A\mathcal E_a(\mathcal C^\infty_{a,{\rm sol}}(M,\C))$ is closed and has finite codimension in $\mathcal C^\infty_{a,{\rm sol}}(M,\C)$. Since
$$
r_M\widetilde{\mathcal N}_{a}A\mathcal E_a(\mathcal C^\infty_{a,{\rm sol}}(M,\C))\subset r_M\widetilde{\mathcal N}_{a}(\mathcal C^\infty_{0}(\widetilde M^{\rm int},\C))\subset \mathcal C^\infty_{a,{\rm sol}}(M,\C),
$$
the intermediate space $r_M\widetilde{\mathcal N}_{a}(\mathcal C^\infty_{0}(\widetilde M^{\rm int},\C))$ is also closed in $\mathcal C^\infty_{a,{\rm sol}}(M,\C)$.
\medskip

Next, we show that the adjoint operator $(r_M \widetilde{\mathcal N}_{a})^*$ has trivial kernel. According to \eqref{decomp of the space of smooth pairs}, each functional on $\mathcal C^\infty_{a,{\rm sol}}(M,\C)$ gives rise to a functional on $\mathcal C^\infty(M,\C)$ that vanishes on $\mathcal C^\infty_{a,{\rm pot}}(M,\C)$. Therefore, the dual of $\mathcal C^\infty_{a,{\rm sol}}(M,\C)$ is
\begin{align*}
\mathcal D&'_{M,\delta_a}(\widetilde M^{\rm int},\C)\\
&=\{[\alpha,f]\in \mathcal D'(\widetilde M^{\rm int},\C):\supp[\alpha,f]\subset M,\,\<[\alpha,f]\,|\,[\tilde \beta,\tilde h]\>=0, \forall[\beta,h]\in \mathcal C^\infty_{a,{\rm pot}}(M,\C)\},
\end{align*}
where $[\tilde \beta,\tilde h]\in\mathcal C^\infty(\widetilde M^{\rm int})$ is any extension of $[\beta,h]$ from $M$ to $\widetilde M^{\rm int}$. Then dual operator of \eqref{r_MN} is
\begin{equation}\label{r_MN*}
(r_M \widetilde{\mathcal N}_a)^*:\mathcal D'_{M,\delta_a}(\widetilde M^{\rm int},\C)\to \mathcal D'(\widetilde M^{\rm int}).
\end{equation}
For all $[\alpha,f]\in \mathcal D'_{M,\delta_a}(\widetilde M^{\rm int},\C)$ and $[\beta,h]\in \mathcal C^\infty(\widetilde M^{\rm int})$
$$
\<(r_M \widetilde{\mathcal N}_a)^*[\alpha,f]\,|\,[\beta,h]\>=\<[\alpha,f]\,|\,(r_M \widetilde{\mathcal N}_a[\beta,h])\widetilde{\,\,\,\,}\>=\<[\alpha,f]\,|\,\widetilde{\mathcal N}_a[\beta,h]\>=\<\widetilde{\mathcal N}_a[\alpha,f]\,|\,[\beta,h]\>.
$$
Hence,
$$
(r_M \widetilde{\mathcal N}_a)^*=\widetilde{\mathcal N}_a|_{\mathcal D'_{M,\delta_a}(\widetilde M^{\rm int},\C)}.
$$

Suppose now that $[\alpha,f]\in \mathcal D'_{M,\delta_a}(\widetilde M^{\rm int},\C)$ is in the kernel of $\widetilde{\mathcal N}_a$. Then from the definition of the space $\mathcal D'_{M,\delta_a}(\widetilde M^{\rm int},\C)$ it follows that
\begin{equation}\label{singsupp of delta_a alpha f}
\sing\supp\delta_a[\alpha,f]\subset\p M.
\end{equation}
Decomposing $[\alpha,f]=[\beta,h]+d_a b$ with $\delta_a[\beta,h]=0$, we have
$$
\delta_a d_a b=\delta_a[\alpha,f].
$$
Since $-\delta_a d_a=-\Delta_g+|a|^2$ is an elliptic operator, \eqref{singsupp of delta_a alpha f} implies
\begin{equation}\label{singsupp b}
\sing\supp b\subset\p M.
\end{equation}
Since $[\alpha,f]$ is supported in $M$, from the decomposition $[\alpha,f]=[\beta,h]+d_a b$ and \eqref{singsupp b} we say
\begin{equation}\label{singsupp beta h}
\sing\supp [\beta,h]\subset M.
\end{equation}
Now consider a smooth function $p$ on $\widetilde M$ equal to $b$ in a neighborhood of $\p \widetilde M$. Then
$$
\mathcal{\widetilde N}_a d_a p=\mathcal{\widetilde N}_a d_a b,
$$
and hence
\begin{equation}\label{N beta h=-Ndp}
\mathcal{\widetilde N}_a[\beta,h]=\mathcal{\widetilde N}_a[\alpha,f]-\mathcal{\widetilde N}_a d_a b=-\mathcal{\widetilde N}_a d_a p.
\end{equation}
This implies that $\mathcal{\widetilde N}_a[\beta,h]$ is smooth in $\widetilde M^{\rm int}$. Now using the fact that $\delta_a[\beta,h]=0$ and \eqref{ABC=1}, we obtain that $[\beta,h]$ is smooth in $\widetilde M^{\rm int}$ and hence according to \eqref{singsupp beta h}, we conclude that $[\beta,h]$ is smooth on $\widetilde M$.
\medskip

By \eqref{N beta h=-Ndp}, we have $\mathcal{\widetilde N}_a\([\beta,h]+d_a p\)=0$ with $[\beta,h]\in\mathcal C^\infty(\widetilde M,\C)$ and $p\in C^\infty(\widetilde M,\C)$. Then $\tilde{\mathcal I}_a\([\beta,h]+d_a p\)=0$ and hence, by the injectivity result \cite[Theorem 1.2]{SaU}, $[\beta,h]+d_a p=d_a q$ for some $q\in C^\infty(\widetilde M,C)$ with $q|_{\p \widetilde M}=0$. This, combined with the decomposition $[\alpha,f]=[\beta,h]+d_a b$, gives
$$
[\alpha,f]=-d_a p+d_a q+d_a b.
$$
Therefore, for every $[\gamma,v]\in\mathcal C^\infty_{a,{\rm sol}}(M,\C)$ we have
$$
\<[\alpha,f]\,|\,[\tilde\gamma,\tilde v]\>=\<[d_a(-p+q+b)\,|\,[\tilde\gamma,\tilde v]\>=-\<(-p+q+b)\,|\,\delta_a[\tilde\gamma,\tilde v]\>,
$$
where $[\tilde\gamma,\tilde v]\in\mathcal C^\infty_0(\widetilde M^{\rm int},\C)$ is any extension of $[\gamma,v]$. By Proposition~\ref{extension of a-solenoidal pairs with compact support}, we can take $[\tilde\gamma,\tilde v]$ to satisfy $\delta_a[\tilde\gamma,\tilde v]=0$. Therefore, $[\alpha,f]$ annihilates $\mathcal C^\infty_{a,{\rm sol}}(M,\C)$. By the definition of $\mathcal D'_{M,\delta_a}(\widetilde M^{\rm int},\C)$ we then have $[\alpha,f]=0$.
\end{proof}

\section{Proofs of Theorems \ref{main} and \ref{thm:charac1}}
Following \cite{PeU2,PSU3}, we start with deriving the appropriate factorization for the operator $P_a$. Suppose $w\in\mathcal S_{a}^\infty(\p_+SM,\C)$. Then $w^\sharp$ is a smooth solution of the transport equation $(X+a)w^\sharp=0$. Applying commutator formula \eqref{[H,X]} to $w^\sharp$, we obtain
$$
-(X+a) H w^\sharp=X_\perp w^\sharp_0+(X_\perp w^\sharp)_0.
$$
Note that $X_\perp w^\sharp_0=\star dw^\sharp_0$. Since $X_\perp=i(\eta_--\eta_+)$, using \cite[Lemma~6.2]{PSU3}, we also have
$$
(X_\perp w^\sharp)_0=i(\eta_-w^\sharp_1-\eta_+w^\sharp_{-1})=\frac{1}{2}\star d(w^\sharp_{-1}+w^\sharp_1).
$$
Therefore,
$$
-2\pi (X+a) H w^\sharp=2\pi \star dw^\sharp_0+\pi \star d(w^\sharp_{-1}+w^\sharp_1).
$$
Applying $I_{a}$ to the above equality and using the expressions for the adjoint of the ray transform
in \eqref{adjoint of I^0} and \eqref{adjoint of I^1}, we deduce
\begin{equation}\label{final equation}
    -2\pi P_{a} w= \cI_{a}[\star d(I_{-\overline{a}}^0)^*(w)\,,\,\star d(I_{-\overline{a}}^1)^*(w)] = \cI_a\ \left[ \begin{smallmatrix} 0 & \star d \\ \star d & 0 \end{smallmatrix} \right] \ \cI_{-\overline{a}}^* w.
\end{equation}

\begin{proof}[Proof of Theorem~\ref{thm:charac1}]
    {\bf Proof of Claim (1).} Suppose that $u=P_{a}w$ for some $w\in \mathcal S_a^\infty(\p_+ SM,\C)$ with $w^\sharp_0=0$. According to \eqref{adjoint of I^0}, $w^\sharp_0=0$ is equivalent to saying $(I^0_{-a})^*(w)=0$,  the factorization \eqref{final equation} shows that $u$ belongs to the range of $I^0_{a}$.
    \medskip
    
    Conversely, suppose $u=I_{a}^0f$ for some $f\in C^\infty(M,\C)$. By basic properties of the Hodge star $\star$, we know that $f=\star(f\,d\Vol_g)$. Since $M$ is simply connected and $f\,d\Vol_g$ is closed, there is a smooth one-form $\alpha$ such that $d\alpha=f\,d\Vol_g$. Recall that $\alpha$ can be written as $\alpha=\alpha^s+dh$ where $\alpha^s$ is solenoidal and $h\in C^\infty(M,\C)$ such that $h|_{\p M}=0$. Then $d\alpha=d\alpha^s$, since $d^2=0$. Therefore, without loss of generality, we can assume $\alpha$ to be solenoidal. Thus, we have $u=I_a^0\star d\alpha$ with $\alpha$ being solenoidal. By Corollary~\ref{main surjectivity result for solenoidal form} there is $w\in\mathcal S_{a}^\infty(\p_+ SM,\C)$ such that $(I_{-\overline{a}}^1)^*(w)=\alpha$ and $(I_{-\overline{a}}^0)^*(w)=0$. Using \eqref{final equation}, we can conclude that
    $$
    u=I_a^0\star d(I_{-\overline{a}}^1)^*(w)=\mathcal I_a[\star d(I_{-\overline{a}}^0)^*(w)\,,\,\star d(I_{-\overline{a}}^1)^*(w)]=P_a w,
    $$
    which finishes the proof of Claim (1). \\

    \noindent{\bf Proof of Claim (2).} Suppose that $u=P_{a}w$ for some $w\in \mathcal S_a^\infty(\p_+ SM,\C)$ such that $w^\sharp_{-1}+w^\sharp_1=dp$ for some $p\in C^\infty(M,\C)$. According to \eqref{adjoint of I^1}, $w^\sharp_{-1}+w^\sharp_1=dp$ is equivalent to saying $(I^1_{-a})^*(w)=dq$ for some $q\in C^\infty(M,\C)$. Then the factorization \eqref{final equation} shows that $u$ belongs to the range of $I^1_{a}$ acting on solenoidal one-forms.
    \medskip
    Conversely, suppose $u=I_{a}^1\star d \varphi$ for some $\varphi\in C^\infty(M,\C)$. Since the Hodge star operator $\star$ is isomorphism between $\Omega^2(M,\C)$ and $C^\infty(M,\C)$, there is a two-form $\omega$ such that $\star\omega= a\varphi$. Since $M$ is simply connected and $\omega$ is closed, there is a smooth one-form $\beta$ such that $\omega=d\beta$. As in the proof of Claim~1, $\beta$ can be taken to be solenoidal, i.e. $\beta=\star dh$. Write $\alpha=\star \beta=-dh$, then one can check that $\delta_{-\overline a}[\alpha,\varphi]=0$. Then by Theorem~\ref{main surjectivity result} there is $w\in\mathcal S_{a}^\infty(\p_+ SM,\C)$ such that $\mathcal I_{-\overline{a}}^*(w)=[\alpha,\varphi]$. Since $d\alpha=0$, using \eqref{final equation}, we can conclude that
    $$
    u=I^1_a\star d\varphi=\mathcal I_a[\star d(I_{-\overline{a}}^0)^*(w)\,,\,\star d(I_{-\overline{a}}^1)^*(w)]=P_a w.
    $$
    According to \eqref{adjoint of I^1}, since $\alpha=-dh$, we have
    $$
    w^\sharp_{-1}+w^\sharp_1=\frac{1}{\pi}(I^1_{-\overline{a}})^*(w)=dq,\quad q=-\frac{1}{\pi}h.
    $$
    Hence, the proof of Claim (2) is complete.
\end{proof}

\section{An injective decomposition of the range of $\cI_a$}\label{sec:viewpoints}

While spaces of pairs are more amenable to the microlocal analysis arguments from the previous sections, inverting $\cI_a$ over pairs requires finding a representative modulo the kernel of $\cI_a$ of $a$-potential pairs. One way to achieve this below is to use a different domain of definition, over which the transform is {\em injective}. We first define the mapping
\begin{align*}
    \dot{C}^\infty(M)\times C^\infty(M) \ni (h,f) \mapsto \cI_a [\star dh, f] \in C^\infty(\partial_+ SM),
\end{align*}
where we have defined the space 
\begin{align}
    \dot{C}^\infty (M) := \left\{ h\in C^\infty(M): \int_{\partial M} h(s)\ ds = 0 \right\},
    \label{eq:dotC}
\end{align}
(note that any other normalization condition setting constants to zero may work) and we establish the following: 

\begin{Theorem}\label{thm:secondrep}
    Let $(M,g)$ a simple surface with boundary and $a\in C^\infty(M)$. Then:
    \begin{itemize}
	\item[$(i)$] The transform $\dot{C}^\infty(M)\times C^\infty(M)\ni (h,f)\mapsto \cI_a[\star dh, f]$ is injective.
	\item[$(ii)$] For any smooth pair $[\alpha,b]\in \cC^\infty(M,\C)$, there exists a unique couple $(h,f)\in \dot{C}^\infty(M)\times C^\infty(M)$ such that $\cI_a [\alpha,b] = \cI_a[\star dh, f]$.
    \end{itemize}
\end{Theorem}

\begin{proof}[Proof of Theorem \ref{thm:secondrep}] {\bf Proof of $(i)$}. Suppose that $(h,f)$ are such that $\cI_a [\star dh, f] = 0$. By solenoidal injectivity of $\cI_a$, this implies that $[\star dh, f] = d_a m = [dm, am]$ for some function $m$ vanishing on $\partial M$. Then the equality $\star dh = dm$ implies that $m$ and $h$ are harmonic. Since $m|_{\partial M} = 0$, then $m=0$ on $M$. In turn, $f = am = 0$ and since $\star dh = 0$ $h$ is constant equal to zero, due to the normalization condition \eqref{eq:dotC}. \medskip

    \noindent{\bf Proof of $(ii)$}. Let $[\alpha,b]$ a smooth pair. Then $\cI_a [\alpha,b] = u|_{\partial_+ SM}$, where $u$ is the solution to 
    \begin{align*}
	Xu + au = - b - \alpha(v) \quad (SM), \qquad u|_{\partial_- SM}. 
    \end{align*}
    Now, $\alpha$ has a unique Hodge decomposition $\alpha = df' + \star d h$ with $f'\in C^\infty(M)$ with $f'|_{\partial M} = 0$ and $h\in \dot{C}^\infty(M)$. As functions on $SM$, this means, $\alpha(v) = Xf' + X_\perp h$, and thus the previous transport equation can be rewritten as 
    \begin{align*}
	X (u+f') + a (u+f') = - (b-af') - X_\perp h,
    \end{align*}
    where the functions $u+f'$ and $u$ agree on $\partial_{\pm} SM$. In particular, $(u+f')|_{\partial_+ SM} = 0$ and 
    \[ \cI_a [\alpha,b] = u|_{\partial_+ SM} = (u+f')|_{\partial_+ SM} = \cI_a[\star dh,b-af']. \]
    Therefore, the couple $(h,b-af')$ provides the desired candidate, whose smoothness comes from elliptic regularity and smoothness of $a$. In addition, such a couple is unique by virtue of $(i)$. Theorem \ref{thm:secondrep} is proved. 
\end{proof}

We now decompose $h$ further. Recall that we define $\ker^k \eta_\pm := \Omega_k\cap \ker\eta_\pm$.

\begin{Lemma}\label{lem:decomp} Any $h\in \dot{C}^\infty(M)$ decomposes into $h = h_0 + h_+ + h_-$, where $h_0\in C^\infty_0(M)$ is unique and $h_\pm \in \ker^0 \eta_\pm$ are unique up to a constant. In particular, $h =0$ if and only if $h_0 = 0$ and $h_+$ and $h_-$ are constant.
\end{Lemma}

\begin{proof} Let $h\in \dot{C}^\infty (M)$ and define $u$ unique harmonic function with $u|_{\partial M} = h|_{\partial M}$. By elliptic regularity and smoothness of $\partial M$, $u\in \dot{C}^\infty(M)$. Let $v$ the unique harmonic conjugate to $u$ satisfying the normalization condition \eqref{eq:dotC}, such that $du = \star d v$. In the sense of functions on $SM$, this is equivalent to saying $Xu = X_\perp v$ which upon using that $X = \eta_+ + \eta_-$ and $X_\perp = \frac{1}{i} (\eta_+ - \eta_-)$, yields
    \[ \eta_+ (u+iv) + \eta_- (u-iv) = 0. \] 
    Projecting onto $\Omega_1$ and $\Omega_{-1}$ gives $\eta_+ (u+iv) = 0$ and $\eta_- (u-iv)$. Therefore, the decomposition follows upon writing
    \[ h = (h-u) + \frac{1}{2} (u+iv) + \frac{1}{2} (u-iv). \] 
    Lemma \ref{lem:decomp} is proved. 
\end{proof}

Upon decomposing $h = h_0 + h_+ + h_-$ as in Lemma \ref{lem:decomp}, and using that $h_\pm \in \ker^0 \eta_\pm$, the data $\cD := \cI_a [\star d h,f]$ looks like 
\begin{align*}
    \cD &= I_a (f + X_\perp h_0 - i \eta_+ h_- + i \eta_- h_+) \\
    &= I_{a}^0 f + I_{a}^\perp h_0 + I_{a}^{+1} (-i\eta_+ h_-) + I_{a}^{-1} (i \eta_- h_+),
\end{align*}
where the equality does not depend on constants added to $h_+$ or $h_-$. From the commutator relation $[\eta_+,\eta_-]= \frac{i}{2} \kappa V$, we can see that $\eta_+ h_- \in \ker^1 \eta_-$ and $\eta_- h_+ \in \ker^{-1} \eta_+$. Upon defining $\omega_1 = -i\eta_+ h_-$ and $\omega_{-1} = i\eta_- h_+$, and in light of Lemma \ref{lem:decomp}, the decomposition $\star d h = \star d h_0 + \omega_1 + \omega_{-1}$ is unique and the left hand side is zero if and only if each summand of the right hand side is zero. Combining this with  Theorem \ref{thm:secondrep}, we arrive at the following conclusion: 

\begin{Lemma}\label{lem:4decomp} For any $\cD\in \text{Range } \cI_a$, there exists a unique quadruple $(f,h_0,\omega_1,\omega_{-1})\in C^\infty(M)\times C^\infty_0(M)\times \ker^1 \eta_- \times \ker^{-1} \eta_+$ such that 
  \begin{align*}
    \cD = I_a^0 f + I_a^\perp h_0 + I_a^{+1} \omega_1 + I_a^{-1} \omega_{-1}. 
  \end{align*} 
  In particular, $\cD=0$ if and only if the entire quadruple vanishes identically. 
\end{Lemma}

\section{Inversion approach} \label{sec:inversion}

As Lemma \ref{lem:4decomp} suggests, since the mapping $(f,h_0,\omega_1,\omega_{-1})\mapsto \cD$ is injective, we expect to write reconstruction formulas for each element of the quadruple, which is the purpose of this section. The remainder is organized as follows: 
\begin{itemize}
  \item In Section \ref{sec:recons1}, we will first show how to reconstruct $\omega_1$ and $\omega_{-1}$ from $\cD$, thereby allowing us to remove the data $I_{a}^{+1} \omega_1 + I_{a}^{-1} \omega_{-1}$ from $\cD$. 
  \item In Section \ref{sec:holomorphization}, as a preparation toward the reconstruction of $(h_0,f)$, we will construct a so-called boundary holomorphization operator, related to the unattenuated transforn $\cI_0$. 
  \item In Section \ref{sec:recons2}, we will then show how to reconstruct $(h_0,f)$ from the remaining data $\cI_a [\star d h_0,f]$ via explicit formulas. 
\end{itemize}

\subsection{Reconstruction of $\omega_1$ and $\omega_{-1}$.} \label{sec:recons1}

Here and below, by $\cO_{\ge k}$ (resp. $\cO_{\le k}$), we denote an element $u\in C^\infty(SM)$ such that $u_p = 0$ for all $p<k$ (resp. all $p>k$). We first recall the following result from \cite{PSU3}, see also \cite{PaternainZhou2015}. 

\begin{Lemma}[Lemma 5.6 in \cite{PSU3}]\label{lem:PSU} Given any $f\in \Omega_m$, there exists $w\in C^\infty(SM,\C)$ such that $Xw = 0$ and $w_m=f$.    
\end{Lemma}
Using Lemma \ref{lem:PSU}, we prove the following: 

\begin{Lemma}\label{lem:goodFirstIntegrals} Let $(M,g)$ simple and $a\in C^\infty(M,\C)$. Then the following statements hold true: 
    \begin{enumerate}
	\item For any $\phi \in \ker^1 \eta_-$, there exists $w = \phi + \cO_{\ge 2}$, solution of $Xw - \abar w = 0$.
	\item For any $\phi \in \ker^{-1} \eta_+$, there exists $w = \phi + \cO_{\le-2}$, solution of $Xw - \abar w = 0$.
    \end{enumerate}    
\end{Lemma}

\begin{proof}[Proof of Lemma \ref{lem:goodFirstIntegrals}] Let $\wh, \wa$ denote smooth, odd solutions of $X\wh = X\wa = \overline{a}$ with $\wh$ holomorphic and $\wa$ antiholomorphic, whose existence is established in \cite[Proposition 4.1]{SaU}. Then $e^\wh$ is a holomorphic solution of $(X - \overline{a})e^\wh = 0$ of the form $e^{\wh} = 1 + \cO_{\ge 1}$ and $e^{\wa} = 1 + \cO_{\le -1}$ is an antiholomorphic solution of $(X - \overline{a})e^\wh = 0$. \\
    {\bf Proof of (1).} For $\phi\in \ker^1\eta_-$, using Lemma \ref{lem:PSU}, there exists $v$ smooth solution of $Xv = 0$ with $v_1 = \phi$. Since $\eta_- v_1 = \eta_- \phi = 0$, then $v' = \sum_{k\ge 1} v_k$ is another smooth solution of $Xv' = 0$ with $v'_1 = v_1 = \phi$. Then setting $w = e^{\wh} v'$ completes the proof. \\
    {\bf Proof of (2).} For $\phi\in \ker^{-1}\eta_+$, using Lemma \ref{lem:PSU}, there exists $v$ smooth solution of $Xv = 0$ with $v_{-1} = \phi$. Since $\eta_+ v_{-1} = \eta_+ \phi = 0$, then $v' = \sum_{k\le -1} v_k$ is another smooth solution of $Xv' = 0$ with $v'_1 = v_1 = \phi$. Then setting $w = e^{\wa} v'$ completes the proof.
\end{proof}

\paragraph{The spaces $L^2(\ker^k \eta_\pm)$.} In the sequel, we denote 
\begin{align*}
    L^2(\ker^k \eta_\pm) := \{ f\in L^2(SM):\ f_p = 0, \quad p\ne k;\ \eta_\pm f = 0  \}.
\end{align*}
These spaces are closed subpaces of $L^2(SM)$, essentially because, using isothermal coordinates, the operators $\eta_\pm$ are $\partial_z, \partial_\zbar$ operators and that $L^2(M)$-limits of solutions of $\partial_\zbar f = 0$ are in fact normal limits (uniform limits on compact subsets of $M$), and thus themselves solutions of $\partial_\zbar f = 0$ (see for instance \cite[Ex. 6 p254]{Stein}). These spaces are therefore Hilbert spaces themselves, admitting complete orthonormal sets. For the sequel, we denote $\{\phi^{\pm 1, (p)}\}_{p=0}^\infty$ orthonormal Hilbert bases of $L^2(\ker^{\pm 1} \eta_\mp)$. Then for any $\phi^{1,(p)}$, we define $w^{1,(p)}$ as in Lemma \ref{lem:goodFirstIntegrals}.(1) and for any $\phi^{-1,(p)}$, we define $w^{-1,(p)}$ according to Lemma \ref{lem:goodFirstIntegrals}.(2). 

We now explain how to reconstruct elements of $\ker^{\pm 1} \eta_{\mp}$ from knowledge of their ray transforms, and notice how these reconstructions pay no heed to the additional terms $f$ and $h_0$.


\begin{Theorem}\label{thm:holoterms} Let $(M,g)$ a simple surface and $a\in C^\infty(M,\C)$. Let $\cD \in \text{Range }\cI_a$ and $(f, h_0, \omega_1,\omega_{-1})$ as in Lemma \ref{lem:4decomp}, then the harmonic one-forms $\omega_1$ and $\omega_{-1}$ can be reconstructed from $\cD = I_{a}^0 f + I_{a}^\perp h_0 + I_{a}^{+1} \omega_1 + I_{a}^{-1} \omega_{-1}$ via the formulas
    \begin{align}
	\omega_1 &= \sum_{p=0}^\infty \dprod{\cD}{w^{1,(p)}|_{\partial_+ SM}}_{L^2_\mu (\partial_+ SM)}\ \phi^{1,(p)}, \label{eq:etaplushminus}\\
	\omega_{-1} &= \sum_{p=0}^\infty \dprod{\cD}{w^{-1,(p)}|_{\partial_+ SM}}_{L^2_\mu (\partial_+ SM)}\ \phi^{-1,(p)}. \label{eq:etaminushplus}
    \end{align}    
\end{Theorem}

\begin{proof} We only prove \eqref{eq:etaplushminus}, as the proof of \eqref{eq:etaminushplus} is similar. \\
    \noindent{\bf Proof of \eqref{eq:etaplushminus}.} Recall that $\cD = u|_{\partial_+ SM}$, where $u$ solves the problem
    \begin{align*}
      Xu + au = -f - X_\perp h_0 - \omega_1 - \omega_{-1} \qquad (SM).
    \end{align*}    
    For any $p\ge 0$, setting $\phi = \phi^{1,(p)}$ and $w = w^{1,(p)}$, we take the $L^2(SM)$ inner product of the transport equation above with $w$ to make appear:
    \begin{align*}
	LHS &= (Xu + au, w)_{SM} = -\dprod{\cD}{w|_{\partial_+ SM}}_{L^2_\mu (\partial_+ SM)} + \cancel{(u, -Xw + \abar w)_{SM}} \\
	RHS &= (-f-X_\perp h_0 - \omega_1 - \omega_{-1}, w)_{SM} = (- \eta_+ h_0 - \omega_1, \phi)_{SM} = -(\omega_1, \phi)_{SM}, 
    \end{align*}
    where the integration by parts $(\eta_+ h_0, \phi)_{SM} = (h_0,\eta_- \phi)_{SM} = 0$ holds with no boundary term since $h_0 |_{\partial M} = 0$. We then arrive at the relation
    \begin{align*}
	\dprod{\cD}{w|_{\partial_+ SM}}_{L^2_\mu (\partial_+ SM)} = (\omega_1, \phi)_{SM}.
    \end{align*}
    Therefore, for $\phi = \phi^{1,(p)}$ above and $w=w^{1,(p)}$ as in Lemma \ref{lem:goodFirstIntegrals}.(1), 
    \begin{align*}
	\dprod{\cD}{w^{1,(p)}|_{\partial_+ SM}}_{L^2_\mu (\partial_+ SM)} = (\omega_1, \phi^{1,(p)})_{SM}, \qquad \forall\ p\ge 0.
    \end{align*}
    Since $\omega_1\in L^2(\ker^1 \eta_-)$, then Bessel's inequality implies that 
    \begin{align*}
	\sum_{p= 0}^\infty |\dprod{\cD}{w^{1,(p)}|_{\partial_+ SM}}_{L^2_\mu (\partial_+ SM)}|^2 = \sum_{p = 0}^\infty |(\omega_1, \phi^{1,(p)})_{SM}|^2 \le \|\omega_1\|^2_{L^2},
    \end{align*} 
    so that the following infinite sum makes sense: 
    \begin{align*}
	\omega_1 = \sum_{p=0}^\infty (\omega_1, \phi^{1,(p)})_{SM}\ \phi^{1,(p)} = \sum_{p=0}^\infty \dprod{\cD}{w^{1,(p)}|_{\partial_+ SM}}_{L^2_\mu (\partial_+ SM)}\ \phi^{1,(p)},
    \end{align*}
    hence \eqref{eq:etaplushminus} is proved. 
\end{proof}

Theorem \ref{thm:holoterms} gives rise to two linear operators $L_{a,\pm 1}:\text{Range }\cI_a \to \ker^{\pm 1} \eta_{\mp}$ satisfying 
\begin{align}
    \begin{split}
      I_a L_{a,+1} ( I_{a}^0 f + I_{a}^\perp h_0 + I_{a}^{+1} \omega_1 + I_{a}^{-1} \omega_{-1} ) &= I_{a}^{+1} \omega_1, \\
      I_a L_{a,-1} ( I_{a}^0 f + I_{a}^\perp h_0 + I_{a}^{+1} \omega_1 + I_{a}^{-1} \omega_{-1} ) &= I_{a}^{-1} \omega_{-1}.	
    \end{split}
    \label{eq:propL1}    
\end{align}
If we then define 
\begin{align}
    P_{a,\pm 1}:\text{Range } \cI_a \to \text{Range } \cI_a, \qquad P_{a,\pm 1} := I_a L_{a,\pm 1},
    \label{eq:Ppm1}
\end{align}
such operators are idempotent on Range $\cI_a$ (i.e., satisfy $P_{a,\pm 1}^2 = P_{a,\pm 1}$). In particular, applying $Id - P_{a, 1} - P_{a,-1}$ to $\cD$ allows to remove $I_{a}^{+1} \omega_1$ and $I_{a}^{-1} \omega_{-1}$ from $\cD$.

\begin{Remark} If the data is not in the range of $\cI_a$ in the first place, the operators $L_{a,\pm 1}$ may pick up some additional components which are in the complement of Range $\cI_a$. This behavior depends on the choice of first integral $w^{\pm 1,(p)}$ for $\phi^{\pm 1,(p)}$. Methods for finding such elements will be the object of future work.
\end{Remark}

\subsection{Holomorphization of solutions to unattenuated transport equations with holomorphic right-hand side} \label{sec:holomorphization}

As a preparation for the reconstruction of $(f,h_0)$, this section focuses on the unattenuated transform 
\begin{align*}
    \dot{C}^\infty(M)\times C^\infty(M) \ni (h,f) \mapsto \cI_0[\star d h, f],
\end{align*}
in particular, how its injectivity allows to produce holomorphic solutions to transport equations with holomorphic right-hand sides, out of any other solution of the same transport problem, via a so-called boundary holomorphization operator.

For conciseness, we will denote $\cI^{0,\perp}[h,f] = \cI_0[\star d h, f]$, and we also denote $I^0$ and $I^1$ the unattenuated transforms over smooth functions and one-forms, and $I^\perp (h) := I^1(\star d h)$ for $h\in \dot{C}^\infty(M)$. The remarks from Section \ref{sec:viewpoints} imply that, while $I^1$ is only solenoidal-injective and has a kernel, $I^\perp$ is injective and both transforms have the same range.

Recall the boundary operators $P_{\pm} := A_-^* H_\pm A_+$ defined in \cite{PeU2}, where $A_+ w = Q_0 w = w_\psi|_{\partial SM}$ and $A_-^* = B_0$ in our current notation. One may simply define $P = A_-^* H A_+ = P_+ + P_-$, where in fact, the operators $P_\pm$ represent the action of $P$ on two orthogonal subspaces of $L^2_\mu(\partial_+ SM)$. In order to clarify this, let us define the {\em antipodal scattering relation} $\alpha_A:\partial_+ SM\to \partial_+ SM$ to be the scattering relation composed with the antipodal map $(x,v)\mapsto(x,-v)$. $\alpha_A$ is clearly an involution of $\partial_+ SM$, and since the measure $\mu\ d\Sigma^2$ is preserved by the pull-back $\alpha_A^*$, the following orthogonal decomposition holds
\begin{align*}
    L^2_\mu (\partial_+ SM) = \V_+ \stackrel{\perp}{\oplus} \V_-, \qquad \V_\pm := \ker (Id \mp \alpha_A^*). 
\end{align*}
Further inspection of symmetries upon applying the operators $A_+$, then $H$, then $A_-^\star$ to functions in $\V_\pm$, shows that the operator $P$, in this decomposition, has the matrix form $P = \left[\begin{smallmatrix} 0 & P_- \\ P_+ & 0 \end{smallmatrix}\right]$. The other facts below are also obvious:
\begin{itemize}
    \item Range $I^0 \subset \V_+$ thus $(I^0)^*(\V_-) = \{0\}$.  
    \item Range $I^\perp \subset \V_-$ thus $(I^\perp)^*(\V_+) = \{0\}$.
\end{itemize}
Now the range characterization \cite[Theorem 4.5]{PeU2} states that $P_-:C_\alpha^\infty (\partial_+ SM)\to C^\infty(\partial_+ SM)$ is surjective on the range of $I^0$ and $P_+:C_\alpha^\infty (\partial_+ SM)\to C^\infty(\partial_+ SM)$ is surjective on the range of $I^\perp$. This justifies the existence of right inverses 
\begin{align*}
    P_+^\dagger&: \text{Range } I^\perp \to C_\alpha^\infty (\partial_+ SM) \cap \V_+, \qquad P_+ P_+^\dagger = Id|_{\text{Range } I^\perp}, \\
    P_-^\dagger&: \text{Range } I^0 \to C_\alpha^\infty (\partial_+ SM) \cap \V_-, \qquad P_- P_-^\dagger = Id|_{\text{Range } I^0}.
\end{align*}
Using the factorizations $2\pi P_+ = I^\perp (I^0)^*$ and $2\pi P_- = -I^0 (I^\perp)^*$, this implies 
\begin{align*}
    2\pi I^\perp h &= 2\pi P_+ P_+^\dagger I^\perp h = I^\perp (I^0)^* P_+^\dagger I^\perp h, \quad \forall h\in \dot{C}^\infty(M), \\
    2\pi I^0 f &= 2\pi P_- P_-^\dagger I^0 f = - I^0 (I^\perp)^* P_-^\dagger I^0 f, \quad \forall f\in C^\infty(M),
\end{align*}
which by injectivity of $I^0$ and $I^\perp$ implies 
\begin{align}
    \begin{split}
	\frac{1}{2\pi} (I^0)^* P_+^\dagger I^\perp h &= h + \text{constant}, \quad \forall h\in \dot{C}^\infty(M), \\
	-\frac{1}{2\pi} (I^\perp)^* P_-^\dagger I^0 f &= f, \quad \forall f\in C^\infty(M).
    \end{split}    
    \label{eq:identities}
\end{align}
Out of the two right-inverses $P_\pm^\dagger$, we may construct a right inverse $P^\dagger$ for $P$, defined on $\text{Range } I^0 \oplus \text{Range }I^\perp = \text{Range } \cI^{0,\perp}$ and $C_\alpha^\infty (\partial_+ SM)$-valued, defined by  
\begin{align*}
    P^\dagger w = P_-^\dagger \frac{1}{2}(Id + \alpha_A^*) w + P_+^\dagger \frac{1}{2} (Id - \alpha_A^*) w, \quad w\in \text{Range } \cI^{0,\perp},
\end{align*}
such that $PP^\dagger = Id$ on $\text{Range } \cI^{0,\perp}$. 

\begin{Theorem}[Holomorphization operator]\label{thm:holomorphization} Let $(M,g)$ a simple Riemannian surface with boundary. There exists a linear boundary operator \[ \Bh:C^\infty(\partial SM)\to C^\infty(\partial_+ SM) \] such that for any function $f\in C^\infty(SM)$ and any solution $u$ of $Xu = -f$ smooth on $SM$, the function $\uh := u - (\Bh (u|_{\partial SM}))_\psi$ satisfies: 
    \begin{enumerate}
	\item If $f = f_{-1} + f_0 + \sum_{k\ge 1} f_k$, then $\uh$ is holomorphic. 
	\item If, additionally, $f_{-1} = 0$, then $\uh_0$ is constant.
    \end{enumerate}
    If $P^\dagger$ denotes any right-inverse for $P$, then such an operator $\Bh$ may be obtained by defining
    \begin{align}
	\Bh h := \frac{1}{2}[(Id - iH)h + i (Id + iH) (A_+ P^\dagger A_-^\star (Id - iH)h)]|_{\partial_+ SM}. 
	\label{eq:Bdef}
    \end{align}
\end{Theorem}

By complex conjugation, we state a corollary of Theorem \ref{thm:holomorphization} without proof, regarding the existence of an anti-holomorphization operator. 

\begin{Corollary}[Anti-holomorphization operator]\label{cor:antiholomorphization}
    With $\Bh$ as in Theorem \ref{thm:holomorphization}, the operator 
    \begin{align*}
    \Ba:C^\infty(\partial SM)\to C^\infty(\partial_+ SM), \qquad \Ba h := \overline{\Bh \overline{h}}, \quad h\in C^\infty (\partial SM),
    \end{align*}
    is such that for any function $f\in C^\infty(SM)$ and any solution $u$ of $Xu = -f$ smooth on $SM$, the function $\ua := u - (\Ba (u|_{\partial SM}))_\psi$ satisfies: 
    \begin{enumerate}
	\item If $f = \sum_{k\le -1} f_k + f_0 + f_1$, then $\ua$ is anti-holomorphic. 
	\item If, additionally, $f_{1} = 0$, then $\ua_0$ is constant.
    \end{enumerate}    
\end{Corollary}

\begin{Remark}
    Using the fact that $\Bh (0) = 0$ and $\Ba(0) = 0$, we recover the statement of \cite[Proposition 5.1]{SaU}: if $u$ solves $Xu = -f$ with $f$ holomorphic (resp. antiholomorphic), and $u|_{\partial SM} = 0$, then $u$ is holomorphic (resp. antiholomorphic) and $u_0 = 0$.
\end{Remark}

\begin{Remark}[Continuity and expliciteness of $\Bh$ and $\Ba$] The continuity of $\Bh$ and $\Ba$ relies heavily on the continuity of $P^\dagger$, for which explicit expressions remain to be found in general. In the case where the surface is such that the operator $Id + W^2$ is invertible (see \cite{PeU2,Monard2015} for a definition of $W$), then such a right-inverse is explicitely constructed in \cite{Monard2015}. This is done by using the factorizations $2\pi P_+ = I^\perp (I^0)^*$ and $2\pi P_- = -I^0 (I^\perp)^*$, and constructing explicit right-inverses for $(I^0)^*$, $(I^\perp)^*$, and inverting $I^0$, $I^\perp$, using the Fredholm equations first derived in \cite{PeU2}. This construction is valid in the case of surfaces with Gaussian curvature close enough to constant, though whether the operator $Id + W^2$ is always invertible on simple surfaces remains open at present.    
\end{Remark}

\begin{proof}[Proof of Theorem \ref{thm:holomorphization}] Let $P^\dagger$ a right-inverse for $P$, let $f\in C^\infty(SM)$ and $u$ a solution of $Xu = -f$.
Write 
\begin{align*}
    u = \frac{1}{2} (u^{(+)} + u^{(-)}), \qquad u^{(\pm)} := (Id \pm iH) u,
\end{align*}
where $(Id - iH) u$ solves the PDE
\begin{align*}
    X(Id - iH) u = (Id-iH) Xu + [X,Id-iH] u = -2f_{-1} - f_0 + i (X_\perp u)_0 + i X_\perp u_0.
\end{align*}
Applying the Hodge decomposition to the one-form $2f_{-1}$, there exists $g\in C^\infty_0(M)$ and $h\in \dot{C}^\infty(M)$ such that $2f_{-1} = Xg + X_\perp h$, in which case the previous equation can be rewritten as
\begin{align*}
    X (u^{(-)} + g) = - (f_0 - i(X_\perp u)_0) - X_\perp (h - iu_0). 
\end{align*}
Upon integrating along geodesics, we make appear
\begin{align}
    A_-^* (u^{(-)}|_{\partial SM}) = A_-^* (u^{(-)} + g)|_{\partial SM} = I^0 [f_0 - i (X_\perp u)_0] + I^{\perp} [h-iu_0], 
    \label{eq:Amstar}
\end{align}
where the right-hand-side belongs to $\text{Range } \cI^{0,\perp}$. 

Define $u' = - i(Id + iH) (P^\dagger A_-^* (u^{(-)}|_{\partial SM}) )_\psi$. $u'$ is holomorphic by construction and we now claim that $(i)$ $u'_0 = -ih -u_0 + C$ (with $C$ a constant), and $(ii)$ $u'$ is another solution of 
\begin{align*}
    Xu' = - (f_0 - i (X_\perp u)_0) - X_\perp (h-iu_0).
\end{align*}
In both claims, we will make use of the observations that, using identities \eqref{eq:identities} and the equality \eqref{eq:Amstar}, we have
\begin{align*}
    \frac{1}{2\pi} (I^0)^* P^\dagger A_-^\star (u^{(-)}|_{\partial SM}) &= \frac{1}{2\pi} (I^0)^* P_+^\dagger I^\perp [h-iu_0 ] = h - iu_0 + iC \quad (C \text{ constant}), \\
    \frac{1}{2\pi} (I^\perp)^* P^\dagger A_-^\star (u^{(-)}|_{\partial SM}) &= \frac{1}{2\pi} (I^\perp)^* P^\dagger_- I^0 [f_0 - i (X_\perp u)_0] = - (f_0 - i(X_\perp u)_0).
\end{align*}
Then proving claim $(i)$ amounts to computing 
\begin{align*}
    u'_0 &= -i \left( (Id + iH) (P^\dagger A_-^* (u^{(-)}|_{\partial SM}) )_\psi \right)_0 \\
    &= -i \left( (P^\dagger A_-^* (u^{(-)}|_{\partial SM}) )_\psi  \right)_0 \\
    &= \frac{-i}{2\pi} (I^0)^* P^\dagger A_-^* (u^{(-)}|_{\partial SM}) \\
    &= -i (h -i u_0 + iC) = - ih - u_0 + C.
\end{align*}
Proving claim $(ii)$ then amounts to computing
\begin{align*}
     Xu' &= - X i(Id + iH) (P^\dagger A_-^* (u^{(-)}|_{\partial SM}) )_\psi \\
     &= XH (P^\dagger A_-^* (u^{(-)}|_{\partial SM}) )_\psi \\
     &= - [H,X] (P^\dagger A_-^* (u^{(-)}|_{\partial SM}) )_\psi \\
     &= - \left( X_\perp (P^\dagger A_-^* (u^{(-)}|_{\partial SM}) )_\psi  \right)_0 - X_\perp \left( (P^\dagger A_-^* (u^{(-)}|_{\partial SM}) )_\psi \right)_0 \\
     &=  \frac{1}{2\pi} I_\perp^* P^\dagger A_-^* (u^{(-)}|_{\partial SM}) -\frac{1}{2\pi} X_\perp I_0^* P^\dagger A_-^* (u^{(-)}|_{\partial SM}) \\
     &= - (f_0 - i (X_\perp u)_0) - X_\perp (h-iu_0 + \cancel{iC}).
\end{align*}

Now that the claims are proved, we use $u'$ to rewrite $u$ as 
\begin{align*}
    u = \frac{1}{2} (u^{(+)} - g + u') + \frac{1}{2} (u^{(-)} + g - u'),
\end{align*}
where the first summand $\uh := \frac{1}{2} (u^{(+)} - g + u')$ is holomorphic by construction, and where the second summand satisfies $X \left( \frac{1}{2}(u^{(-)} + g - u')\right) = 0$, so that it is equal to some $h_\psi$, where $h = \frac{1}{2} (u^{(-)} + g - u')|_{\partial_+ SM} = \Bh (u|_{\partial SM})$ by construction. So Claim 1 is proved. As for Claim 2, if $f_{-1} = 0$, then the Hodge decomposition above becomes $h = g = 0$, and using claim $(i)$, we read
\[ 2 \uh_0 = (u^{(+)} - g + u')_0 = u_0 - g - ih - u_0 + C = C. \]
Thus Theorem \ref{thm:holomorphization} is proved. 
\end{proof}

\subsection{Reconstruction of $f$ and $h_0$.} \label{sec:recons2}

In light of Section \ref{sec:recons1}, given the quadruple $(f, h_0,\omega_1,\omega_2)$, it is possible to extract $\cI_{a}[\star d h_0, f]$ from $\cD = I_{a}^{0} f + I_{a}^{\perp} h_0 + I_{a}^{+1} \omega_1 + I_{a}^{-1} \omega_{-1}$ via the processing $\cI_{a}[\star d h_0, f] = (Id - P_{a,1} - P_{a,-1})\cD$. From this data, and using the results of the previous section, let us now provide explicit inversion formulas for $(h_0,f)$. Recall the statement of Theorem \ref{thm:intro} stated in Section \ref{sec:main}.  

\begin{Theorem}\label{thm:inversionfh0} Let $(M,g)$ a simple surface and $a\in C^\infty(M,\C)$. Define $\wh$ and $\wa$ smooth holomorphic and antiholomorphic, odd, solutions of $X\wh = X\wa = -a$, and let $\Bh$ and $\Ba$ as in Theorem \ref{thm:holomorphization} and Corollary \ref{cor:antiholomorphization}. Then the functions $(h_0,f)\in C^\infty_0(M)\times C^\infty(M)$ can be reconstructed from data $\cI := \cI_{a} [\star dh_0,f]$ (extended by zero on $\partial_- SM$) via the following formulas:
    \begin{align*}
	f &= -\eta_+ \Dh_{-1} - \eta_- \Da_1 - \frac{a}{2} \left( \Dh_0 + \Da_0 + i(g_+ - g_-) \right), \\
	h_0 &= \frac{1}{2} (g_+ + g_-) - \frac{i}{2} (\Dh_0 - \Da_0),
    \end{align*}
    where we have defined $\Dh := e^\wh (\Bh(\cI e^{-\wh}|_{\partial SM}))_\psi$, $\Da := e^\wa (\Ba(\cI e^{-\wa}|_{\partial SM}))_\psi$, and where $g_\pm \in \ker^0 \eta_\pm$, uniquely characterized by their boundary conditions
    \begin{align*}
	g_+|_{\partial M} = -i (\cI - \Dh|_{\partial SM})_0, \qquad g_-|_{\partial M} = i (\cI - \Da|_{\partial SM})_0. 
    \end{align*}
\end{Theorem}

\begin{proof}[Proof of Theorem \ref{thm:inversionfh0}]
    Let $e^{-\wh}$ and $e^{-\wa}$ holomorphic and antiholomorphic integrating factors for $a$ (in particular, $\wh$ is an odd, holomorphic solution of $X\wh = -a$ and $\wa$ is an odd, antiholomorphic solution of $X\wa = -a$), so as to obtain 
    \begin{align*}
	X (u e^{-\wh}) = - (f + X_\perp h_0) e^{-\wh} = - b (x,v), 
    \end{align*}
    where $b$ is of the form $b_{-1} + b_0 + \cO_{\ge 1}$ with, in particular, $b_{-1} = -\frac{1}{i} \eta_- h_0$. Thanks to Theorem \ref{thm:holomorphization}, defining $v:= ue^{-\wh}$, the function $\vh = v - (\Bh(v|_{\partial SM}))_\psi$ is holomorphic and satisfies
    \begin{align*}
	X \vh = -b. 
    \end{align*}
    Then defining $\uh := e^\wh \vh = u - e^\wh (\Bh( ue^{-\wh}|_{\partial SM}))_\psi = u - \Dh$, $\uh$ solves the equation 
    \begin{align}
	X \uh + a \uh = -f-X_\perp h_0.    
	\label{eq:transuh}
    \end{align}
    Similarly using the antiholomorphic integrating factor, the function $\ua = u - e^{\wa} (\Ba( ue^{-\wa}|_{\partial SM}))_\psi = u-\Da$ is antiholomorphic and solves
    \begin{align}
	X \ua + a \ua = -f-X_\perp h_0.    
	\label{eq:transua}
    \end{align}
    Projecting \eqref{eq:transuh} onto $H_{-1}$ and \eqref{eq:transua} onto $H_1$, we obtain 
    \begin{align*}
	\eta_- \uh_0 = \frac{1}{i} \eta_- h_0 &\qquad \Leftrightarrow \qquad \eta_- (h_0 -i\uh_0) = 0, \\
	\eta_+ \ua_0 = -\frac{1}{i} \eta_+ h_0 &\qquad \Leftrightarrow \qquad \eta_+ (h_0 + i\ua_0) = 0.
    \end{align*}
    This implies the relations: 
    \begin{align}
	\begin{split}
	    h_0 - i\uh_0 &= g_+ \in \ker^0 \eta_+, \\
	    h_0 + i\ua_0 &= g_- \in \ker^0 \eta_-.	
	\end{split}
	\label{eq:relh0}    
    \end{align}
    which, since $h_0$ vanishes at the boundary, completely determines $g_\pm$ from their boundary values, which are in turn determined from the boundary values of $\uh$ and $\ua$, in turn determined by the data. Taking the half-sum, we obtain
    \begin{align*}
	h_0 = \frac{1}{2} (g_+ + g_-) - \frac{i}{2} (\uh_0 - \ua_0) = \frac{1}{2} (g_+ + g_-) - \frac{i}{2} (\Dh_0 - \Da_0),
    \end{align*}
    where the right-hand side is completely determined by data. On to the determination of $f$, we project the equation $Xu + au = -f-X_\perp h_0$ onto $H_0$ to make appear
    \begin{align*}
	f = -\eta_+ u_{-1} - \eta_- u_1 - a u_0,
    \end{align*}
    and show how to determine each term from the data. Since $\uh$ is holomorphic, then $\uh_{-1} = 0 = u_{-1} - \Dh_{-1}$, so $u_{-1} = \Dh_{-1}$. Since $\ua$ is antiholomorphic, $\ua_1 = 0 = u_1 - \Da_1$, so $u_1 = \Da_1$. Finally, 
    \begin{align*}
	u_0 = \frac{1}{2} (\uh_0 + \Dh_0 + \ua_0 + \Da_0) \stackrel{\eqref{eq:relh0}}{=} \frac{1}{2} (\Dh_0 + \Da_0) + \frac{i}{2} (g_+ - g_-).  
    \end{align*}
    We arrive at the following formula for $f$ 
    \begin{align*}
	f = -\eta_+ \Dh_{-1} - \eta_- \Da_1 - \frac{a}{2} \left( \Dh_0 + \Da_0 + i(g_+ - g_-) \right). 
    \end{align*}
    Theorem \ref{thm:inversionfh0} is proved.
\end{proof}

Theorem \ref{thm:inversionfh0} gives rise to two linear operators 
\begin{align*}
    L_{a,0} &: \cI^{0,\perp}_a (C^\infty_0(M)\times C^\infty(M)) \to C^\infty(M), \\
    L_{a,\perp} &: \cI^{0,\perp}_a (C^\infty_0(M)\times C^\infty(M)) \to C^\infty_0(M),
\end{align*}
such that 
\begin{align*}
    I_a L_{a,0} (I_a^0 f + I_a^\perp h_0) = I_a^0 f, \quad\text{and}\quad I_a L_{a,\perp} (I_a^0 f + I_a^\perp h_0) = I_a^\perp h_0.
\end{align*}
If we then define $P_{a,0}, P_{a,\perp} : \text{Range } \cI_a^{0,\perp} \to \text{Range } \cI_a^{0,\perp}$, by  
\begin{align}
    P_{a,0} := I_a L_{a,0}(Id - P_1 - P_{-1}), \qquad P_{a,\perp} := I_a L_{a,\perp} (Id - P_1 - P_{-1}),
    \label{eq:P0perp}
\end{align}
such operators are idempotent, projections of $\text{Range } \cI_a $ onto $I_a^0 (C^\infty(M))$ and $I_a^\perp (C_0^\infty(M))$, respectively. 

\section*{Acknowledgements} FM acknowledges partial funding from NSF grant DMS-1634544. GU was partly supported by NSF and a Walker Family Professorship, FiDiPo Professorship and a Si-Yuan Professorship.

\end{document}